\documentclass[a4paper,12pt]{amsart}

\usepackage[margin=2.5cm]{geometry}

\usepackage{amsthm,amsmath,amsfonts,amssymb}
\usepackage{enumitem}
\usepackage{graphicx} 
\usepackage{longtable}
\usepackage{multirow}
\usepackage{tabu}  

\usepackage{hyperref}
\hypersetup{
	colorlinks   = true, 
	urlcolor     = black, 
	linkcolor    = black, 
	citecolor   = black }

\newcommand{\z}[0]{{\mathbb Z}}
\newcommand{\n}[0]{{\mathbb N}}
\newcommand{\co}[0]{{\mathbb C}}
\newcommand{\Lmb}[0]{\mathbb{L}}
\newcommand{\C}[0]{{\mathcal C}}
\newcommand{\lbeta}[0]{\bar{\beta}}
\newcommand{\lcm}[0]{\text{lcm}}

\newtheorem{theo}{Theorem}[section]
\newtheorem{prop}[theo]{Proposition}
\newtheorem{cor}[theo]{Corollary}
\newtheorem{lem}[theo]{Lemma}

\theoremstyle{definition}
\newtheorem{rem}{Remark}[section]
\newtheorem{ex}{Example}[section]

\title[The motivic zeta function of a space monomial curve]{The motivic Igusa zeta function of a space monomial curve with a plane semigroup}

\author[H. Mourtada]{Hussein Mourtada}
\address[H. Mourtada]{Universit\'e de Paris, Sorbonne Universit\'e, CNRS, Institut de Math\'ematiques de Jussieu-Paris Rive Gauche, F-75013 Paris, France.} 

\email{hussein.mourtada@imj-prg.fr}

\author[W. Veys]{Willem Veys}
\author[L. Vos]{Lena Vos}
\address[W. Veys \and L. Vos]{KU Leuven, Departement Wiskunde, Celestijnenlaan 200B, bus 2400, 3001 Leuven, Belgium.}
\email{wim.veys@kuleuven.be, lena.vos@kuleuven.be}

\thanks{The second author is partially supported by the Research Foundation - Flanders (FWO) project G.0792.18N. The third author is supported by a PhD Fellowship of the Research Foundation - Flanders (no. 71587).}

\begin{document}

\begin{abstract} 
In this article, we compute the motivic Igusa zeta function of a space monomial curve that appears as the special fiber of an equisingular family whose generic fiber is a complex plane branch. To this end, we determine the irreducible components of the jet schemes of such a space monomial curve. This approach does not only yield a closed formula for the motivic zeta function, but also allows to determine its poles. We show that, while the family of the jet schemes of the fibers is not flat, the number of poles of the motivic zeta function associated with the space monomial curve is equal to the number of poles of the motivic zeta function associated with a generic curve in the family. 
\end{abstract}


\maketitle 

\section*{Introduction}

The history of the \emph{motivic Igusa zeta function} goes back to the seventies when Igusa studied the \emph{$p$-adic Igusa zeta function}, which is related to the classical problem in number theory of computing the number of solutions of congruences. More precisely, the original Igusa zeta function counts, for a non-constant polynomial ${f\in \z[x_1,\ldots,x_n]}$ and a prime number $p$, the $\z/(p^{m+1}\z)$-points of $X=\{f=0\}$, when $m$ varies in $\n$. It was introduced by Weil~\cite{W}, and its basic properties, such as rationality, were first investigated by Igusa~\cite{I}. In analogy with the $p$-adic zeta function, Denef and Loeser~\cite{DL2} introduced the \lq more general\rq\ motivic Igusa zeta function in which $f\in \mathbb{C}[x_1,\ldots,x_n]$ is a complex polynomial, and the $\z/(p^{m+1}\z)$-points of $X = \{f=0\}$ are replaced by its $\co[t]/(t^{m+1})$-points. It is more general in the sense that the earlier zeta function can be obtained from the motivic one. \\

 The space of $\co[t]/(t^{m+1})$-points of $X = \{f=0\}$ or, equivalently, of morphisms $\text{Spec} (\co[t]/(t^{m+1})) \rightarrow X$ has a natural scheme structure; it is denoted by $X_m$ and called the \emph{$m$th jet scheme} of $X$. Geometrically, if we consider $X$ in the affine space $\co^n$, the space $X_m$ can be thought of as the moduli space parameterizing germs of curves in $\co^n$  which have a \lq contact\rq\ with $X$ larger than $m$. Simple invariants of the $X_m$ (e.g., their irreducible components and their dimensions) encode deep information on the singularities of $X$, see for instance~\cite{M1},\cite{M3} and~\cite{Mu1}. In terms of these jet schemes, the motivic Igusa zeta function $Z_X^{mot}(T)$ associated with $X$ (or with $f$) can be written as \[Z_X^{mot}(T) = 1 - \frac{1-T}{T}J_X(T),\] where $J_X(T)$ is the Poincar\'e series \[J_X(T) := \sum_{m \geq 0}[X_m](\Lmb^{-n}T)^{m+1} \in \mathcal{M}_{\co}[[T]].\] Here, $\mathcal{M}_{\co}$ is a localization of the \emph{Grothendieck ring} of complex varieties, and $[X_m]$ and $\Lmb$ are the classes of $X_m$ and of the affine line $\mathbb{C}$ in this Grothendieck ring, respectively. Clearly, this expression also makes sense when $X$ is any subscheme of $\co^n$ given by some ideal $I$ in $\mathbb{C}[x_1,\dots,x_n]$, instead of just a hypersurface. Furthermore, the motivic zeta function turns out to be a rational function in $T$, and it is natural to study its poles. \\

The motivic Igusa zeta function for one polynomial $f \in \co[x_1, \ldots, x_n]$ can also be expressed in terms of an \emph{embedded resolution of singularities} of $f$; the analogous expression for an ideal is in terms of a \emph{principalization} of the ideal. This formula in the hypersurface case can be found in~\cite{DL2}, and its generalization to ideals is mentioned in~\cite{VZ}. It is the most classical way to compute the motivic zeta function and allows to determine a complete list of candidates poles of this zeta function. However, it is in general very difficult to calculate a principalization and to verify whether the candidate poles are actual poles; usually, \lq most\rq\ of the candidates are in fact no actual poles. Therefore, in this article, in order to determine the motivic zeta function and its poles, we will compute the above Poincar\'e series, based on the structure of the jet schemes. \\

We will apply this approach to a class of monomial curves that naturally appear as the special fibers of (equisingular) families of curves whose generic fibers are isomorphic to some irreducible plane (germ of a) curve. More precisely, let $\C:=\{f=0\}\subset (\co^2,0)$ be a germ of a complex plane curve defined by an irreducible series $f\in\co[[x_0,x_1]]$ with $f(0) = 0$, and let \[\nu_{\C}: R:=\frac{\co[[x_0,x_1]]}{(f)} \longrightarrow \n\] be the associated valuation, where $\nu_{\C}(h)=(f,h)_0$ is the local intersection multiplicity of the curve $\C$ and the curve $\{h=0\}$. The semigroup $\Gamma(\C):=\{\nu_{\C}(h) \mid h \in R\setminus \{0\}\}\subset \n$ is finitely generated, and we can identify a unique minimal system of generators $(\bar{\beta}_0,\ldots,\bar{\beta}_g)$ of $\Gamma(\C)$. Let $(Y,0) \subset (\co^{g+1},0)$ be the image of the monomial map $M:(\co,0) \rightarrow (\co^{g+1},0)$ given by $M(t)=(t^{\bar{\beta}_0},t^{\bar{\beta}_1},\ldots,t^{\bar{\beta}_g}).$ It is an irreducible curve with the \lq plane\rq\ semigroup $\Gamma(\C)$ as its semigroup and it is the special fiber of a flat family $\eta:(\chi,0) \subset (\co^{g+1} \times \co,0)\rightarrow (\co,0)$ whose generic fiber is isomorphic to $\C$. We call $Y$ the \emph{monomial curve associated with $\C$}, and the explicit equations defining $Y$ in $\co^{g+1}$ are of the form
	\[\left\{ 
	\begin{array}{r c l l}
	x_1^{n_1} & - & x_0^{n_0}  &  = 0 \\
	x_2^{n_2}  &- & x_0^{b_{20}}x_1^{b_{21}} &= 0  \\
	& \vdots & &\\
	x_g^{n_g} &- & x_0^{b_{g0}}x_1^{b_{g1}}\cdots x_{g-1}^{b_{g(g-1)}} & = 0,\\
 	 \end{array}
	\right.\]
where $n_i > 1$ and $b_{ij} \geq 0$ are integers that are defined in terms of $(\bar{\beta}_0,\ldots, \bar{\beta}_g)$. \\

We will first study the jet schemes $Y_m$ of $Y$ for every $m \in \n$. Because $Y_0 = Y$ and the natural morphism $\pi_m:Y_m\rightarrow Y$ induces a trivial fibration over $Y \setminus \{0\}$ with fiber $\co^m$, the interesting information is concentrated in the fibers $\pi_m^{-1}(0)$ for $m\geq 1$. From the irreducible components of these fibers, we easily find the decomposition into irreducible components of the whole jet scheme $Y_m$ for $m\geq 1$, see Theorem~\ref{ThmIrrCom} and Corollary~\ref{CorIrrComWhole}, respectively. In addition, we can associate with this decomposition a natural graph which is similar to the one in~\cite{CM},~\cite{M1} and~\cite{M3}, and which we use to encode the computation of the motivic Igusa zeta function of $Y$, see Figure~\ref{GeneralPicture}. With this point of view, we are able to compute a closed formula for the motivic zeta function in Theorem~\ref{TheoremMotFunction}: 
\begin{align*}
	Z^{mot}_Y(T) =  & \frac{1-(\Lmb-1)\Lmb^{-(g+1)} - \Lmb^{-(g+1)}T}{1-\Lmb^{-g}T} + \frac{P_0(T)}{1- \Lmb^{-\nu_1}T^{N_1}} \\
	&+ \sum_{i=1}^{g-1}\frac{P_i(T)}{(1-\Lmb^{-\nu_i}T^{N_i})(1-\Lmb^{-\nu_{i+1}}T^{N_{i+1}})} - \frac{(\Lmb-1)\Lmb^{-(\nu_g + g + 1)}(1-T)T^{N_g}}{(1-\Lmb^{-g}T)(1-\Lmb^{-\nu_g}T^{N_g})},
\end{align*}
where $P_i(T)$ for $i = 0, \ldots, g-1$ are concrete polynomials with coefficients in the ring $\z[\Lmb, \Lmb^{-1}]$, and $(N_i,\nu_i)$ for $i = 1, \ldots, g$ are couples of known positive integers with \[\frac{\nu_i}{N_i} = \frac{1}{n_i\lbeta_i} \bigg(\sum_{l=0}^i\lbeta_l - \sum_{l=1}^{i-1}n_l\lbeta_l\bigg) + (i-1)+ \sum_{l=i+1}^g \frac{1}{n_l}.\] Except for some \lq small\rq\ concrete cases, we do not see how one can obtain such a formula using a principalization, see also Remark~\ref{RemPrinc}. Furthermore, we obtain only $g+1$ candidate poles: \[\Lmb^g, \qquad \Lmb^{\frac{\nu_i}{N_i}}, \quad i=1, \ldots, g.\] Using residues and the related \emph{topological Igusa zeta function}, we prove that, contrary to formulas that one could obtain using a principalization, all these candidate poles are actual poles, see Theorem~\ref{ThmPoles}. We also get the \emph{log canonical threshold} of $Y\subset \co^{g+1}$ given by $\frac{\nu_1}{N_1} = \sum_{l=0}^g \frac{1}{n_l}.$ Note that the number of poles of the motivic zeta function of $Y$ is equal to the number of poles of the motivic zeta function of the plane branch $\C$. This implies that the motivic zeta function associated with the special fiber of the family $\eta: (\chi,0) \rightarrow (\co,0)$ has the same number of poles as the motivic zeta function associated with the generic fiber. This is a fascinating result as the induced family on the level of jet schemes is not flat. More precisely, let $S := (\co,0)$ and consider, for every $m \in \n$, the \emph{relative $m$th jet scheme} $((\chi,0)/S)_m$ of $\eta:(\chi,0) \rightarrow S$ with the natural morphism $\eta_m:((\chi,0)/S)_m \rightarrow S$, whose fibers are isomorphic to the $m$th jet schemes of the fibers of $\eta$. Then, although the family $\eta$ is equisingular (in particular, flat), we show in Theorem~\ref{TheoremNotFlat} that the family $\eta_m$ is not flat for $m$ large enough. We would like to point out that, in the hypersurface case, an equisingular family of hypersurfaces does induce a flat family on the jet schemes (with their reduced structures)~\cite[Theorem 3.4]{LA}. \\

The poles of the motivic Igusa zeta function associated with a complex polynomial $f \in \co[x_1, \ldots, x_n]$ are the subject of an intriguing open problem, the so-called \emph{monodromy conjecture}, which relates number theoretical invariants and topological invariants of $f$. Roughly speaking, it predicts a relation between the poles of the motivic zeta function and the action of the monodromy of $f$, seen as a function $f:\co^n\rightarrow \co$, on the cohomology of its Milnor fiber at some point $x \in X\subset \co^n.$ For an ideal $I$, one can state the \emph{generalized monodromy conjecture} in which Verdier monodromy replaces the classical monodromy. To date, both the classical and the generalized conjecture have only been proven in full generality for polynomials and ideals in two variables, see~\cite{L} and~\cite{VV}, respectively. In higher dimension, there are various partial results in the hypersurface case (we refer to the introduction of~\cite{BV} for a list of references), while in the non-hypersurface case, the most general result so far is a proof for monomial ideals~\cite{HMY}. The results in this article make the first part of this conjecture for monomial curves of the above type very explicit; the study of their monodromy part and proof of the monodromy conjecture can be found in~\cite{MVV}. In other words, these two articles together solve the conjecture for an interesting class of binomial ideals in arbitrary dimension. \\

The article is organized as follows. We assume $\n$ to be the set of non-negative integers. We begin in Section~\ref{SpaceCurve} with introducing the curves $Y \subset \co^{g+1}$ in which we are interested. Section~\ref{JetMotivic} consists of a brief discussion of the jet schemes and motivic zeta function associated with an affine variety. In Section~\ref{Jets}, we determine the irreducible components of the jet schemes $Y_m$ for $m\in \n$ and show that the induced family on the jet schemes is not flat for most $m$. Based on the structure of $Y_m$, we compute the motivic zeta function of $Y$, find its $g+1$ candidate poles, and provide some examples in Section~\ref{Motivic}. Finally, in Section~\ref{PolesMotivic}, we prove that all candidate poles are actual poles. 

\section{Space monomial curves with plane semigroups}\label{SpaceCurve}

In this section, we introduce the type of singularities that we will consider in this article. They arise as (equisingular) deformations of germs of irreducible plane curves. We begin with an irreducible series $f\in \co[[x_0,x_1]]$ in two variables over the complex numbers satisfying $f(0)=0$. We denote by $\C:=\{f=0\}\subset (\co^2,0)$ the germ at the origin of the curve defined by $f.$ We can assume, modulo a change of variables, that the curve $\{x_0=0\}$ is transversal to $\C$ and that the curve $\{x_1=0\}$ has maximal contact (among smooth curves) with $\C.$ To $\C,$ one can relate a valuation \[\nu_\C: \frac{\co[[x_0,x_1]]}{(f)} \setminus \{0\} \longrightarrow \n: h \mapsto \dim_\co \frac{\co[[x_0,x_1]]}{(f,h)}.\] We denote by $\Gamma(\C)$ the \emph{semigroup associated with} $\nu_\C$: \[\Gamma(\C):=\left\{\nu_\C(h)~\Big\vert~ h \in \frac{\co[[x_0,x_1]]}{(f)} \setminus \{0\}\right\}\subset \n.\] Then, $\Gamma(\C)$ is a finitely generated semigroup with which we can associate the following data~\cite[Chapter II]{Z}: 
\begin{enumerate}
	\item the unique minimal system of generators $(\bar{\beta}_0,\ldots,\bar{\beta}_g)$ of $\Gamma(\C)$ with ${\lbeta_0 < \cdots < \lbeta_g}$ and $\gcd(\lbeta_0, \ldots, \lbeta_g) = 1$ (gcd being the greatest common divisor);
	\item the integers $e_i:=\gcd(\bar{\beta}_0,\ldots,\bar{\beta}_i)$ for $i=0,\ldots,g$, where $e_0 = \lbeta_0, e_g=1$ and $e_0 > \cdots > e_g$; and
	\item the integers $n_i:=\frac{e_{i-1}}{e_i}\geq 2$ for $i=1,\ldots,g.$
\end{enumerate} 
It can be shown that the integer $n_i\bar{\beta}_i$ for $i=1,\ldots,g$ belongs to the semigroup generated by $\bar{\beta}_0,\ldots,\bar{\beta}_{i-1}$, see for instance~\cite{A} or~\cite[Lemma 2.2.1]{T1}. Hence, for $i = 1, \ldots, g,$ there exists a unique system of non-negative integers $\{b_{ij}\}_{ 0 \leq j < i}$ such that \[b_{ij}<n_j ~~\text{for}~~j\not=0 ~~~\text{and}~~~ n_i\bar{\beta}_i=b_{i0}\bar{\beta}_0+\cdots +b_{i(i-1)}\bar{\beta}_{i-1};\] the uniqueness follows from the inequalities $b_{ij}<n_j$. For notational reasons, we introduce $n_0:=b_{10}$. Additionally, one can show that $n_i\lbeta_i < \lbeta_{i+1}$ for all $i = 1, \ldots, g-1$. It is also worth noting that $e_i = n_{i+1}\cdots n_g$ for $i = 0, \ldots, g-1$, that $n_0 > n_1 \geq 2$, and that $n_0 = \frac{\lbeta_1}{e_1}$ and $n_1 = \frac{\lbeta_0}{e_1}$ are coprime. Furthermore, one can choose a \emph{system of approximate roots} or a \emph{minimal generating sequence} $(x_0, \ldots, x_g)$ of $\nu_{\C}$, where $x_i \in \co[[x_0,x_1]]$ such that $\nu_{\C}(x_i)=\bar{\beta}_i$ for $i=0,\ldots,g$, see for example~\cite{AM},~\cite{M2},~\cite{S} or~\cite{T1}. For $i=0$ and $i=1$, the condition $\nu_{\C}(x_i)=\bar{\beta}_i$ is equivalent to the assumptions that we put above on the variables $x_0$ and $x_1$, respectively. These elements satisfy identities of the form \[vx_{i+1}=x_i^{n_i}-c_ix_0^{b_{i0}}\cdots x_{i-1}^{b_{i(i-1)}}- \sum_{\gamma=(\gamma_0,\ldots,\gamma_i)} c_{i,\gamma }vx_0^{\gamma_0}\cdots x_i^{\gamma_i}, \qquad  i = 0, \ldots, g,\] where $v=1, x_{g+1}=0, c_i \in \co\setminus \{0\},  c_{i,\gamma }\in \co,$ $0\leq \gamma_j <n_j$ for $1\leq j \leq i$, and $\sum_{j=0}^i \gamma_j\bar{\beta}_{j}>n_i\bar{\beta}_{i}.$ \\

The above equations with $v = 1$ allow us to embed $\C$ as a complete intersection in $(\co^{g+1},0)$ with coordinates $x_0,\ldots,x_g$. Making $v$ vary in $(\co,0)$ defines a family $(\chi,0)\subset (\co^{g+1}\times \co,0)$ of germs of curves, which is \emph{equisingular} for instance in the sense that all branches in the family have the same semigroup $\Gamma(C)$. We denote by $\eta:(\chi,0)\rightarrow (\co,0)$ the restriction to $(\chi,0)$ of the projection on the second factor $(\co^{g+1}\times \co,0)\rightarrow (\co,0).$ The special fiber $Y:=\eta^{-1}(0)$ is the curve which is of interest to us and is defined by the equations 
\begin{equation} \label{EquationsY}
	\left\{ 
	\begin{array}{r c l l}
	f_1:= x_1^{n_1} & - & c_1x_0^{n_0}  &  = 0 \\
	f_2:= x_2^{n_2}  &- & c_2x_0^{b_{20}}x_1^{b_{21}} &= 0  \\
	& \vdots & &\\
	f_g := x_g^{n_g} &- & c_gx_0^{b_{g0}}x_1^{b_{g1}}\cdots x_{g-1}^{b_{g(g-1)}} & = 0.\\
    \end{array}
 	\right.
\end{equation}

After a change of variables in the coordinates $x_0,\ldots,x_g,$ we can assume that every $c_i$ for $i=1,\ldots,g$ is equal to $1;$ the coefficients $c_i$ are important to see that any irreducible plane curve is a (equisingular) deformation of a curve of type $Y.$ For simplicity, throughout this article, we will consider $c_i=1$ for $i=1,\ldots,g.$ 

\begin{rem}
It is worth mentioning that the above embedding of $\C$ in $(\co^{g+1},0)$ as a complete intersection is Newton non-degenerate in the sense of~\cite{AGS} and~\cite{Te1}. Such an embedding always exists for a singularity in characteristic $0$~\cite{Te2}, and is conjectured to exist in positive characteristic~\cite{T2}.
\end{rem}

The curve $Y$ is called the \emph{monomial curve associated with} $\C$ because it is the image in $(\co^{g+1},0)$ of the monomial map $M:(\co,0) \rightarrow (\co^{g+1},0)$ given by \[M(t)=(t^{\bar{\beta}_0},t^{\bar{\beta}_1},\ldots,t^{\bar{\beta}_g}).\] In particular, $Y$ is an irreducible (germ of a) curve with its semigroup equal to the \lq plane\rq\ semigroup $\Gamma(\C)$, see~\cite{T1} for these and other properties of $Y$. Finally, note that, even though $Y$ has been defined as a deformation of a germ, we can consider the global curve in $\co^{g+1}$ defined by the above polynomial (actually, binomial) equations of $Y$. This is still an irreducible curve, and from now on, we will denote by $Y \subset \co^{g+1}$ this global curve and refer to it as a \emph{(space) monomial curve}.

\section{Jet schemes and motivic Igusa zeta function}\label{JetMotivic}

This section provides a short introduction to the jet schemes and the motivic Igusa zeta function associated with an affine variety. By a \emph{(complex) variety}, we mean a reduced, separated scheme of finite type over $\co$, which is not necessarily irreducible. Let $X=V(I)\subset \co^{g+1}$ be an affine variety defined by an ideal $I=(f_1,\ldots,f_r)\subset \co[x_0,\ldots,x_g]$.\\

For every $m \in \mathbb{N}$, the \emph{$m$th jet scheme} of $X$ is the $\co$-scheme $X_m$ whose $\co$-points are \[ X_m(\co)=\{\text{Spec}(\co[t]/(t^{m+1}))\longrightarrow X\}.\] It immediately follows that $X_0=X.$ For general $m$, one can derive the defining equations of $X_m$ in its natural ambient space $\co^{(g+1)(m+1)}$ as follows. Let $x_i^{(j)}$ for $i=0,\ldots,g$ and $j=0,\ldots,m$ be the coordinates in $\co^{(g+1)(m+1)}$. We will denote by $\underline{x}^{(j)}$ the $(g+1)$-tuple $(x_0^{(j)},\ldots,x_g^{(j)})$ and by $\underline{x}(t)$ the element 
\begin{align*}
	\underline{x}(t)&:=\underline{x}^{(0)}+\underline{x}^{(1)}t+\cdots +\underline{x}^{(m)}t^m \\
	& :=(x_0^{(0)}+x_0^{(1)}t+\cdots +x_0^{(m)}t^m,\ldots,x_g^{(0)}+x_g^{(1)}t+\cdots +x_g^{(m)}t^m)
\end{align*}
in $\co[\underline{x}^{(j)};j=0,\ldots,m][t]/(t^{m+1})$. For $k=1,\ldots,r$ and $l=0,\ldots,m,$ let $F_k^{(l)}\in \co[\underline{x}^{(j)};$ $j=0,\ldots,l]$ be the elements which satisfy the identity 
\begin{equation} \label{IdentityF}
	f_k(\underline{x}(t))=F_k^{(0)}+F_k^{(1)}t+\cdots+F_k^{(m)}t^m~\text{mod}~(t^{m+1}). 
\end{equation}
Then, we have \[X_m=\text{Spec}\frac{\co[\underline{x}^{(j)};j=0,\ldots,m]}{(F_k^{(l)};k=1,\ldots,r;~l=0,\ldots,m)}.\]  For $m,p \in \n$ with $m\geq p,$ there is a natural map $\pi_{m,p}:X_m\rightarrow X_p$ induced by the truncation map $\co[t]/(t^{m+1})\rightarrow \co[t]/(t^{p+1}).$ We will put $\pi_m$ for $\pi_{m,0}.$ Note that for $m,p,q \in \n$ with $m\geq p\geq q,$ we have $\pi_{p,q}\circ\pi_{m,p}=\pi_{m,q}.$ \\

In order to define the motivic zeta function associated with $X$, we first give a brief introduction to the Grothendieck ring of complex varieties and fix some notation. Let $\text{Var}_{\co}$ be the category of complex varieties. The \emph{Grothendieck group} $K_0(\text{Var}_{\co})$ is the abelian group generated by the symbols $[V]$ for $V \in \text{Var}_{\co}$ with the following two relations: $[V] = [W]$ for isomorphic $V$ and $W$, and $[V] = [W] + [V \setminus W]$ for $W$ closed in $V$. By using the multiplication $[V]\cdot [W] := [V \times W]$, the Grothendieck group becomes a commutative ring with $1 := [\text{Spec}~\co]$ as unit element, and we still denote the \emph{Grothendieck ring} by $K_0(\text{Var}_{\co})$. We write $\Lmb := [\co]$ for the class of the affine line and $\mathcal{M}_{\co} := K_0(\text{Var}_{\co})[\Lmb^{-1}]$ for the ring obtained by inverting $\Lmb$. 

\begin{rem}\label{RemGrothRing}
For a constructible subset $W$ of a variety $V$ (i.e., $W$ is a finite union of locally closed subvarieties of $V$), we can define its class in $K_0(\text{Var}_{\co})$ as follows. First, we can always write $W$ as a finite \emph{disjoint} union $W_1 \sqcup \cdots \sqcup W_r$ of locally closed subvarieties of $V$. Then, one can show that $[W] := \sum_{i=1}^r [W_i]$ is well-defined as element in $K_0(\text{Var}_{\co})$. In particular, this definition implies for a locally trivial fibration $p:V \rightarrow B$ with fiber $F$ that $[V] = [B]\cdot[F]$.
\end{rem}

Note that for every $m \in \n$, a point $\gamma \in \text{Spec}~\co[\underline{x}^{(j)}; j = 0, \ldots, m]$ corresponds to a jet $\gamma(t) = (\gamma_0(t), \ldots, \gamma_g(t)) \in (E[t]/(t^{m+1}) )^{g+1}$ for some field extension $E$ of $\co$. We will often also denote this jet by $\gamma := \gamma(t)$. Hence, we can define \[ord_t (f_k(\gamma)) := ord_t(f_k(\gamma(t))),\] for $k=1, \ldots, g$, and \[\mathcal{X}_m  := \Big\{\gamma \in \text{Spec}~\co[\underline{x}^{(j)}; j = 0, \ldots, m] ~\Big\vert~ \min_{k=1, \ldots, r} ord_t(f_k(\gamma))=m\Big\}.\] For each $m$, the set $\mathcal{X}_m$ is a locally closed subvariety of $\text{Spec}~\co[\underline{x}^{(j)}; j = 0, \ldots, m]$, and it thus defines a class $[\mathcal{X}_m]$ in the Grothendieck ring. \\

With this notation, the \emph{(global) motivic Igusa zeta function} associated with the variety $X$ (or with the ideal $I$) is the formal power series \[Z^{mot}_X(T) := \Lmb^{-(g+1)}\sum_{m\geq 0}[\mathcal{X}_m](\Lmb^{-(g+1)}T)^m \in \mathcal{M}_{\co}[[T]].\] There is also a local version where $\mathcal{X}_m$ is replaced by $\mathcal{X}_{m,0}$ consisting of those $\gamma \in \mathcal{X}_m$ with $x_i^{(0)} = 0$ for $i = 0, \ldots, g$ or, equivalently, with associated jet $\gamma(t)$ having the origin $0$ as center (i.e., $\gamma(0)=0$). For $X$ defined by one polynomial (i.e., $I = (f)$), the motivic zeta function was introduced and shown to be rational by Denef and Loeser in~\cite{DL2}. The definition for general ideals can be found in~\cite{VZ}. \\ 

To find the motivic zeta function $Z^{mot}_Y(T)$ associated with a space monomial curve $Y \subset \co^{g+1}$, we will not compute the above series directly. Instead, we will compute the Poincar\'e series \[J_Y(T) := \sum_{m \geq 0}[Y_m](\Lmb^{-(g+1)}T)^{m+1},\] where $[Y_m] \subset \co^{(g+1)(m+1)}$ is the $m$th jet scheme of $Y$. This is well-defined because of the fact that $[Y_m] = [(Y_m)_{red}]$ in $K_0(\text{Var}_{\co})$. Using the relations $[\mathcal{Y}_0] = \Lmb^{g+1}- [Y_0]$ and $[\mathcal{Y}_m] =  \Lmb^{g+1}[Y_{m-1}]-[Y_m]$ for $m \geq 1$, it is not hard to see that the Poincar\'e series is related to $Z^{mot}_Y(T)$ by \[Z^{mot}_Y(T) = 1-\frac{1-T}{T}J_Y(T).\]

\begin{rem}
One often considers the more natural Poincar\'e series $\sum_{m \geq 0} [Y_m]T^m$. We choose to work with the above series $J_Y(T)$ because the factor $\Lmb^{-(g+1)(m+1)}$ implies, in some sense, that we need to look for the codimension of $Y_m$ in $\co^{(g+1)(m+1)}$.
\end{rem}

\section{Jet schemes of space monomial curves whose semigroup is plane} \label{Jets}

In this section, we will study the jet schemes $Y_m$ of a space monomial curve $Y\subset \co^{g+1}$ defined in Section~\ref{SpaceCurve}. The information is mainly concentrated in the fibers $\pi_m^{-1}(0)$ of $\pi_m:Y_m\rightarrow Y$ for $m \geq 1$; indeed, $Y_0 = Y$ and the restriction of $\pi_m$ to $\pi_m^{-1}(Y\setminus \{0\})$ for $m\geq 1$ is a trivial fibration whose fibers are isomorphic to $\co^{m}$. Furthermore, since such a monomial curve for $g = 1$ is a plane curve $Y \subset \co^2$ with one Puiseux pair for which the structure of $\pi^{-1}_m(0)$ has been studied in~\cite[Corollary 4.4]{M1}, we will be concerned with the case $g \geq 2$. To determine the irreducible components of the fibers $\pi^{-1}_m(0)$ with their unique reduced subscheme structure, we first consider the reduced structure $\pi_m^{-1}(0)_{red}$. In the next proposition, we see that $\pi_m^{-1}(0)_{red}$ for $1 \leq m\leq n_0n_1$ is irreducible and rather easy to understand, where we denote the integer part of a rational number $\frac{a}{b}$ by $[\frac{a}{b}]$.

\begin{prop}\label{PropStructureBasis}
\begin{enumerate}[wide, labelindent=0pt]
 	\item For $m\in \n$ satisfying $0<m< n_0n_1$, we have \[\pi_m^{-1}(0)_{red}=\text{Spec} \frac{\co[\underline{x}^{(j)};j=0,\ldots,m]}{(x_i^{(0)},\ldots,x_i^{([\frac{m}{n_i}])};i=0,\ldots,g)}.\]
 	\item The fiber $\pi_{n_0n_1}^{-1}(0)_{red}$ is given by \[\text{Spec} \frac{\co[\underline{x}^{(j)};j=0,\ldots,n_0n_1]}{(x_0^{(0)},\ldots,x_0^{(n_1-1)}, x_1^{(0)}, \ldots, x_1^{(n_0-1)},{x_1^{(n_0)}}^{n_1}-{x_0^{(n_1)}}^{n_0},x_i^{(0)},\ldots,x_i^{([\frac{n_0n_1}{n_i}])};i=2,\ldots,g)}.\]
\end{enumerate}
\end{prop}

The following lemma is used in the proof of Proposition~\ref{PropStructureBasis}. 

\begin{lem}\label{LemmaStructureBasis}
For $i=2,\ldots,g,$ we have $b_{i0}>n_0.$
\end{lem}

\begin{proof}
Fix $i \in \{2, \ldots, g\}$. We prove the inequality in the lemma by contradiction; assume that $b_{i0}\leq n_0.$ On the one hand, we have \[b_{i0}\bar{\beta}_0+b_{i1}\bar{\beta}_1+\cdots+b_{i(i-1)}\bar{\beta}_{i-1}\leq n_0\bar{\beta}_0+n_1\bar{\beta}_1+\cdots+n_{i-1}\bar{\beta}_{i-1}-(\bar{\beta}_1 +\cdots+\bar{\beta}_{i-1}),\]
where we used that $b_{ij}\leq n_j-1$ for $j=1,\ldots,i-1$. On the other hand, by repeatedly using that $\bar{\beta}_l>n_{l-1}\bar{\beta}_{l-1}$ and that $n_l\geq 2$, we find that 
\begin{align*}
	n_i\bar{\beta}_i & > n_i(n_{i-1}\bar{\beta}_{i-1}) \\
	&\geq n_{i-1}\bar{\beta}_{i-1}+n_{i-1}\bar{\beta}_{i-1} \\
	&> n_{i-1}\bar{\beta}_{i-1}+n_{i-1}(n_{i-2}\bar{\beta}_{i-2}) \\
	&\geq n_{i-1}\bar{\beta}_{i-1}+n_{i-2}\bar{\beta}_{i-2}+n_{i-2}\bar{\beta}_{i-2} \\
	& ~~ \vdots \\
	& \geq n_{i-1}\bar{\beta}_{i-1}+n_{i-2}\bar{\beta}_{i-2}+\cdots+n_{2}\bar{\beta}_{2}+n_{1}\bar{\beta}_{1} + n_{1}\bar{\beta}_{1} \\
	& = n_{i-1}\bar{\beta}_{i-1}+n_{i-2}\bar{\beta}_{i-2}+\cdots+n_{2}\bar{\beta}_{2}+n_{1}\bar{\beta}_{1}+n_0\bar{\beta}_0,
\end{align*}
where the last equality follows from $n_0\bar{\beta}_0=b_{10}\bar{\beta}_0=n_1\bar{\beta}_1.$  The two families of inequalities give
\begin{align*}
	b_{i0}\bar{\beta}_0+b_{i1}\bar{\beta}_1+\cdots+b_{i(i-1)}\bar{\beta}_{i-1} & \leq n_0\bar{\beta}_0+n_1\bar{\beta}_1+\cdots+n_{i-1}\bar{\beta}_{i-1}-(\bar{\beta}_1 +\cdots+\bar{\beta}_{i-1}) \\
	&< n_0\bar{\beta}_0+n_1\bar{\beta}_1+\cdots+n_{i-1}\bar{\beta}_{i-1} \\
	&< n_i\bar{\beta}_i.
\end{align*}
This contradicts the equality $b_{i0}\bar{\beta}_0+b_{i1}\bar{\beta}_1+\cdots+b_{i(i-1)}\bar{\beta}_{i-1}=n_i\bar{\beta}_i$ from Section~\ref{SpaceCurve}.
\end{proof}
 
\begin{proof} [Proof of proposition~\ref{PropStructureBasis}] 
We begin with proving the first part; assume that ${0<m<n_0n_1}$. Recall that $f_1,\ldots,f_g$ are the defining equations of $Y$ given in~\eqref{EquationsY}, and that a closed point $\gamma \in \text{Spec}~\co[\underline{x}^{(j)};j=0,\ldots,m]$ corresponds to a jet that we also call $\gamma=\gamma(t)=(\gamma_0(t),\ldots,\gamma_g(t))$ with \[\gamma_i(t)=\sum\limits_{l=0}^mx_i^{(l)}t^l,\] where $x_i^{(l)}$ are the coordinates of $\gamma.$ With this notation, $\gamma \in \pi_m^{-1}(0)_{red} $ if and only if the coordinates $x_i^{(0)}$ for $i=0,\ldots,g$ are zero (i.e., the center of $\gamma$ is the origin) and $\text{ord}_t(f_k(\gamma))\geq m+1$ for $k=1,\ldots,g.$ From Proposition 4.1 in~\cite{M1} applied to the curve given by the equation $x_1^{n_1} -  x_0^{n_0}=0,$ we find that \[x_i^{(l)} = 0~, \qquad i=0,1~;~l=0,\ldots,\Big[\frac{m}{n_i}\Big].\] Taking into consideration that $m<n_0n_1$, one can see that these conditions are the only conditions coming from the equation $f_1 = x_1^{n_1} -  x_0^{n_0}.$ In particular, assuming that the center of $\gamma$ is the origin, the condition $\text{ord}_t(f_1(\gamma))\geq m+1$ implies that $\text{ord}_t(\gamma_0)\geq [\frac{m}{n_0}]+1.$ We will prove that $\text{ord}_t(\gamma_0)\geq [\frac{m}{n_0}]+1$ in turn implies for every $i=2,\ldots,g$ that $\text{ord}_t(f_i(\gamma))\geq m+1$ is equivalent to $\text{ord}_t(\gamma_i)\geq [\frac{m}{n_i}]+1.$ Recall that $f_i=x_i^{n_i}-x_0^{b_{i0}}\cdots x_{i-1}^{b_{i(i-1)}}.$ Since, by Lemma~\ref{LemmaStructureBasis}, we have $b_{i0}>n_0$ and since $\text{ord}_t(\gamma_0)\geq [\frac{m}{n_0}]+1,$ we obtain 
\begin{align*}
	\text{ord}_t(x_0^{b_{i0}}\cdots x_{i-1}^{b_{i(i-1)}}(\gamma)) = \text{ord}_t(\gamma_0^{b_{i0}}\cdots \gamma_{i-1}^{b_{i(i-1)}}) &\geq \text{ord}_t(\gamma_0^{b_{i0}}) \\
	& \geq\text{ord}_t(\gamma_0^{n_0+1}) \\
	& \geq (n_0+1)\Big(\Big[\frac{m}{n_0}\Big]+1\Big) \\
	& \geq m+1.
\end{align*}
Hence, $\text{ord}_t(f_i(\gamma))\geq m+1$ if and only if $\text{ord}_t(x_i^{n_i}(\gamma))=\text{ord}_t(\gamma_i^{n_i})\geq m+1.$ As the latter is equivalent to $\text{ord}_t(\gamma_i)\geq [\frac{m}{n_i}]+1$, this ends the proof of the first part. \\

We now prove the second part of the proposition. Let $\gamma\in \text{Spec} ~\co[\underline{x}^{(j)};j=0,\ldots,n_0n_1]$ be again associated with a $(g+1)$-tuple $\gamma=\gamma(t).$ It follows from the first part that $\gamma\in \pi_{n_0n_1}^{-1}(0)_{red}$ implies that \[x_i^{(l)} = 0 ~, \qquad i=0,\ldots,g~;l=0,\ldots,\Big[\frac{n_0n_1-1}{n_i}\Big].\] Noticing that $f_1$ is weighted homogeneous of degree $n_0n_1$ if we give $x_0$ the weight $n_1$ and $x_1$ the weight $n_0$, we can write 
\begin{align*}
 	f_1(\gamma) & \equiv f_1(t^{n_1}(x_0^{(n_1)}+x_0^{(n_1+1)}t+\cdots+x_0^{(n_0n_1)}t^{n_0n_1-n_1}), \\
 	& \phantom{{}=f_1(}t^{n_0}(x_1^{(n_0)}+x_1^{(n_0+1)}t+\cdots+x_1^{(n_0n_1)}t^{n_0n_1-n_0}))\\
 	& \equiv t^{n_0n_1}f_1(x_0^{(n_1)}+x_0^{(n_1+1)}t+\cdots+x_0^{(n_0n_1)}t^{n_0n_1-n_1}, \\
 	& \phantom{{} = t^{n_0n_1}f_1(} x_1^{(n_0)}+x_1^{(n_0+1)}t+\cdots+x_1^{(n_0n_1)}t^{n_0n_1-n_0})\\
 	& \equiv t^{n_0n_1}f_1(x_0^{(n_1)},x_1^{(n_0)}) \\
 	& \equiv t^{n_0n_1}({x_1^{(n_0)}}^{n_1}-{x_0^{(n_1)}}^{n_0})~~~\text{mod}~~~(t^{n_0n_1+1}).
\end{align*}
Thus, modulo that $\gamma$ is centered at the origin and that $\text{ord}_t(f_1(\gamma))\geq n_0n_1,$ the condition $\text{ord}_t(f_1(\gamma))\geq n_0n_1+1$ is equivalent to ${x_1^{(n_0)}}^{n_1}-{x_0^{(n_1)}}^{n_0}=0.$ For $i = 2, \ldots, g$, again modulo that $\gamma$ is centered at the origin and that $\text{ord}_t(f_i(\gamma))\geq n_0n_1$, one can now see as in the first part of the proof that $\text{ord}_t(f_i(\gamma))\geq n_0n_1+1$ if and only if $x_i^{[\frac{n_0n_1}{n_i}]}=0.$ This proves the second part of the proposition.
\end{proof}

The same reasoning as in the proof of Proposition~\ref{PropStructureBasis} gives us the following corollary.

\begin{cor}\label{CorStructureBasis}
 Let $l\in \n$. For $m \in \n$ such that $ln_0n_1<m< (l+1)n_0n_1$, we have \[\pi_{m,ln_0n_1}^{-1}(\{x_0^{(0)}=x_0^{(1)}=\cdots=x_0^{(ln_1)}=0\})_{red}=\text{Spec} \frac{\co[\underline{x}^{(j)};j=0,\ldots,m]}{(x_i^{(0)},\ldots, x_i^{([\frac{m}{n_i}])};i=0,\ldots,g)}.\] The ideal defining the embedding of $\pi_{(l+1)n_0n_1,ln_0n_1}^{-1}(\{x_0^{(0)}=x_0^{(1)}=\cdots=x_0^{(ln_1)}=0\})_{red}$ in the affine space $\text{Spec}~ \co[\underline{x}^{(j)};j=0,\ldots,(l+1)n_0n_1]$ is generated by \[x_0^{(0)},\ldots,x_0^{((l+1)n_1-1)}, x_1^{(0)},\ldots,x_1^{((l+1)n_0-1)}, {x_1^{((l+1)n_0)}}^{n_1}-{x_0^{((l+1)n_1)}}^{n_0},\] \[x_i^{(0)},\ldots,x_i^{([\frac{(l+1)n_0n_1}{n_i}])}; i=2,\ldots,g.\]
\end{cor}

From Proposition~\ref{PropStructureBasis}, we know that $\pi_{m}^{-1}(0)_{red}$ for $0 < m< n_0n_1$ is irreducible and isomorphic to an affine space. We also know that $\pi_{n_0n_1}^{-1}(0)_{red}$ is the product of an affine space and a hypersurface defined by the equation ${x_1^{(n_0)}}^{n_1}-{x_0^{(n_1)}}^{n_0}=0.$ We can stratify $\pi_{n_0n_1}^{-1}(0)_{red}\subset \co^{(g+1)(n_0n_1+1)}$ as follows: \[\pi_{n_0n_1}^{-1}(0)_{red}=(\pi_{n_0n_1}^{-1}(0)_{red}\cap \{x_0^{(n_1)}=0\}) \sqcup (\pi_{n_0n_1}^{-1}(0)_{red}\cap \{x_0^{(n_1)}\not=0\}).\] For $n_0n_1<m\leq 2n_0n_1,$ Corollary~\ref{CorStructureBasis} shows that $\pi_{m,n_0n_1}^{-1}(\pi_{n_0n_1}^{-1}(0)_{red}\cap \{x_0^{(n_1)}=0\}) = \pi_{m,n_0n_1}^{-1}(\{x_0^{(0)}=x_0^{(1)}=\cdots=x_0^{(n_1)}=0\})$ is irreducible and has a rather simple geometry. Therefore, noticing that $\pi_m=\pi_{n_0n_1}\circ\pi_{m,n_0n_1}$, we need to study the inverse image under $\pi_{m,n_0n_1}$ of $\pi_{n_0n_1}^{-1}(0)_{red}\cap \{x_0^{(n_1)})\not=0\}$ to obtain a stratification of $\pi_{m}^{-1}(0)_{red}$ in two strata of which we understand the geometry. For $2n_0n_1 < m$, Corollary~\ref{CorStructureBasis} does not immediately give us an easy description of $\pi_{m,n_0n_1}^{-1}(\{x_0^{(0)}=x_0^{(1)}=\cdots=x_0^{(n_1)}=0\})$. However, because $\pi_{m,n_0n_1}=\pi_{2n_0n_1,n_0n_1}\circ\pi_{m,2n_0n_1}$, understanding this inverse image boils down to understanding the inverse image under $\pi_{m,2n_0n_1}$ of $\pi_{2n_0n_1,n_0n_1}^{-1}(\{x_0^{(0)}=x_0^{(1)}=\cdots=x_0^{(n_1)}=0\})$, which we again stratify in two strata corresponding to $x_0^{(2n_0)} = 0 $ and $x_0^{(2n_0)} \neq 0$. For general $m \geq 1$, the stratification of $\pi_{(k+1)n_0n_1,kn_0n_1}^{-1}(\{x_0^{(0)}=x_0^{(1)}=\cdots=x_0^{(kn_1)}=0\})_{red}$ for $k\geq 0$ as 
\begin{align*}
	&(\pi_{(k+1)n_0n_1,kn_0n_1}^{-1}(\{x_0^{(0)}=x_0^{(1)}=\cdots=x_0^{(kn_1)}=0\})_{red}\cap \{x_0^{((k+1)n_1)}=0\}) \\
	&\sqcup (\pi_{(k+1)n_0n_1,kn_0n_1}^{-1}(\{x_0^{(0)}=x_0^{(1)}=\cdots=x_0^{(kn_1)}=0\})_{red}\cap \{x_0^{((k+1)n_1)}\not=0\})
\end{align*}
    yields the following stratification of $\pi_{m}^{-1}(0)_{red}:$ let $l\in \n$ such that $ln_0n_1<m\leq(l+1)n_0n_1,$ then
\begin{equation}\label{Stratification}
    \pi^{-1}_m(0)_{red} = \Big(\bigsqcup_{k=1}^l D_{m,k}\Big) \sqcup B_m,
\end{equation}
where 
\begin{alignat*}{2}
	&D_{m,k} &&:= \pi_{m,kn_0n_1}^{-1}(\{x_0^{(0)}=x_0^{(1)}=\cdots=x_0^{(kn_1-1)}=0\}\cap \{x_0^{(kn_1)}\not=0\})_{red}, \\
	&B_m &&:= \pi_{m,ln_0n_1}^{-1}(\{x_0^{(0)}=x_0^{(1)}=\cdots=x_0^{(ln_1)}=0\})_{red}.
\end{alignat*}
This stratification will allow us to determine the irreducible components of $Y_m$ and will be crucial for our computation of the motivic zeta function associated with $Y$. It is important to notice, as we will see later, that some of the above strata may be empty. Furthermore, $B_m$ is a closed irreducible subvariety of $\pi_m^{-1}(0)_{red}$ of which Corollary~\ref{CorStructureBasis} provides the geometric structure. In particular, we know its codimension in $\mathbb{C}^{(g+1)(m+1)}$.

\begin{cor}\label{CorDimBm}
Let $l \in \n$. For $m\in\mathbb{N}$ such that $ln_0n_1<m<(l+1)n_0n_1,$ the codimension of $B_m$ in $\mathbb{C}^{(g+1)(m+1)}$ is equal to \[g+1+\sum_{i=0}^{g} \Big[\frac{m}{n_i}\Big].\] For $m=(l+1)n_0n_1,$ the codimension of $B_m$ in $\mathbb{C}^{(g+1)(m+1)}$ is equal to \[g+ (l+1)(n_0+n_1)+\sum_{i=2}^{g} \Big[\frac{(l+1)n_0n_1}{n_i}\Big].\]
\end{cor}

Still, we need to understand the geometry of the strata $D_{m,k}$ for $k = 1, \ldots, l$, which are locally closed subvarieties of $\pi_m^{-1}(0)_{red}$. We begin with introducing some useful notations. Firstly, for $k \geq 1$ and $m \geq kn_0n_1$, let $\C_{m,k}:= \overline{D_{m,k}}$ be the Zariski closure of $D_{m,k}$ in $\pi_m^{-1}(0)_{red}.$ Secondly, for $k \geq 1$, let $j(k) \in \n$ be defined by 
 	\[j(k) := \left\{ 
 	\begin{array}{cl}
 	2 & \text{ if } n_2 \nmid k \\
 	\text{max} _{l \in \n} \{n_2\cdots n_{l-1} \mid k\} & \text{ otherwise.}
 	\end{array} 
 	\right. \]
 Note that $2 \leq j(k) \leq g+1$. For $1\leq i<j(k)$ and for $m\in \n$ satisfying  \[\frac{kn_i\bar{\beta}_i}{e_1}\leq m<\frac{kn_{i+1}\bar{\beta}_{i+1}}{e_1},\] where, by convention, $\bar{\beta}_{g+1}:=+\infty$, we define 
\begin{equation}\label{cik(m)}
 	c_{i,k}(m):= k(n_0+n_1)+\sum_{l=2}^i\frac{k\bar{\beta}_l}{e_1}+ \sum_{l=1}^i\Big(m-\frac{kn_l\bar{\beta}_l}{e_1}+1\Big)+\sum_{l=i+1}^g\Big(\Big[\frac{m}{n_l}\Big]+1\Big).
\end{equation}
 Finally, let $Y^i$ for $i=1,\ldots,g$ be the complete intersection curve defined in $\co^{i+1}$ by the first $i$ equations $f_1,\ldots,f_i$ of the $g$ defining equations~\eqref{EquationsY} of $Y$. Note that $Y^g=Y$ and that $Y^i \setminus \{0\} \simeq \co \setminus \{0\}$ for all $i = 1,\ldots, g$. 

\begin{prop}\label{PropDmk} Let $k \geq 1$ and $1\leq i < j(k)$. For $m\in \n$ with $\frac{kn_i\bar{\beta}_i}{e_1}\leq m<\frac{kn_{i+1}\bar{\beta}_{i+1}}{e_1},$ the stratum $D_{m,k}$ is isomorphic to \[(Y^i\setminus \{0\})~\times~ \co^{(g+1)(m+1)-c_{i,k}(m)-1}\simeq (\co\setminus \{0\}) ~\times~ \co^{(g+1)(m+1)-c_{i,k}(m)-1}.\]
In particular, $\C_{m,k}$ is irreducible and its codimension in $\co^{(g+1)(m+1)}$ is equal to $c_{i,k}(m).$ 
\noindent For $m\geq \frac{kn_{j(k)}\bar{\beta}_{j(k)}}{e_1},$ we have that $D_{m,k}=\emptyset.$
\end{prop}

Before giving a proof of this proposition, we show the next lemma.

\begin{lem}\label{LemmaDmk}
	Let $i,j\in \n$ be such that $i+1\leq j \leq g.$ We have \[b_{j0}\bar{\beta}_0+\cdots+b_{ji}\bar{\beta}_i\geq n_{i+1}\bar{\beta}_{i+1},\] and the inequality is strict if $i+1 < j.$ 
\end{lem}

\begin{proof}
For $i+1 = j,$ the inequality is an equality and there is nothing to prove. Assume that $i+1<j.$ On the one hand, we have \[b_{j0}\bar{\beta}_0+\cdots+b_{ji}\bar{\beta}_i = n_j\bar{\beta}_j-b_{j(i+1)}\bar{\beta}_{i+1}-\cdots-b_{j(j-1)}\bar{\beta}_{j-1} \geq n_j\bar{\beta}_j-n_{i+1}\bar{\beta}_{i+1}-\cdots-n_{j-1}\bar{\beta}_{j-1}.\]
The inequality follows from the fact that $b_{jl}<n_l$ for $l=1,\ldots,j-1.$ On the other hand, as in the proof of Lemma~\ref{LemmaStructureBasis}, we have \[n_j\bar{\beta}_j >n_{j-1}\bar{\beta}_{j-1}+n_{j-2}\bar{\beta}_{j-2}+\cdots+n_{i+1}\bar{\beta}_{i+1}+n_{i+1}\bar{\beta}_{i+1}.\]
The two series of inequalities give the strict inequality in the lemma.
\end{proof}

\begin{proof}[Proof of Proposition~\ref{PropDmk}] 
We still denote by $f_1,\ldots,f_g$ the defining equations of $Y$ and by $F_h^{(l)}$ for $h=1,\ldots,g$ and $l\in \n$ the polynomials defined from $f_h$ by the identity~\eqref{IdentityF} in Section~\ref{JetMotivic}. We prove that for $m\in \n$ with $\frac{kn_i\bar{\beta}_i}{e_1}\leq m<\frac{kn_{i+1}\bar{\beta}_{i+1}}{e_1}$ for some $1\leq i < j(k)$, the ideal defining the embedding of $D_{m,k}$ in $\text{Spec} ~\co[\underline{x}^{(j)};j=0,\ldots,m]_{x_0^{(kn_1)}}$ is generated by \[x_r^{(0)},\ldots,x_r^{(\frac{k\bar{\beta}_r}{e_1}-1)},~\mathcal F_h^{(l_h)}, ~x_s^{(0)},\ldots,x_s^{([\frac{m}{n_s}])};\] \[r = 0, \ldots, i;~ h=1,\ldots,i;~l_h=0, \ldots, m - \frac{kn_h\bar{\beta}_h}{e_1};~s=i+1,\ldots,g,\] where $\mathcal F_h^{(l_h)} := F_h^{(l_h)}\Big(x_0^{(\frac{k\lbeta_0}{e_1})}, \ldots, x_h^{(\frac{k\lbeta_h}{e_1})},\ldots,x_0^{(\frac{k\lbeta_0}{e_1} + l_h)}, \ldots,x_h^{(\frac{k\lbeta_h }{e_1}+l_h)}\Big).$ More precisely, for $h=1,\ldots,i,$ \[\mathcal F_h^{(0)} = {x_h^{(\frac{k\bar{\beta}_h}{e_1})}}^{n_h}-{x_0^{(\frac{k\bar{\beta}_0}{e_1})}}^{b_{h0}}\cdots {x_{h-1}^{(\frac{k\bar{\beta}_{h-1}}{e_1})}}^{b_{h(h-1)}}\] and for $l=1,\ldots,m-\frac{kn_h\bar{\beta}_h}{e_1},$ \[\mathcal F_h^{(l)} = \alpha_l{x_h^{(\frac{k\bar{\beta}_h}{e_1})}}^{n_h-1}x_h^{(\frac{k\bar{\beta}_h}{e_1}+l)}-H_l\Big(x_0^{(\frac{k\bar{\beta}_0}{e_1})},\ldots,x_h^{(\frac{k\bar{\beta}_h}{e_1})},\ldots,x_0^{(\frac{k\bar{\beta}_0}{e_1}+l)},\ldots,x_{h-1}^{(\frac{k\bar{\beta}_{h-1}}{e_1}+l)}, x_h^{(\frac{k\bar{\beta}_h}{e_1}+l-1)}\Big)\] for some $\alpha_l\in \co \setminus \{0\}$ and $H_l$ a polynomial. \\

The proof is by induction on $i.$ We begin with the case $i=1;$  let $kn_0n_1 \leq m < \frac{kn_2\lbeta_2}{e_1}$. As in Proposition~\ref{PropStructureBasis}, a closed point $\gamma \in \text{Spec} ~\co[\underline{x}^{(j)};j=0,\ldots,m]$ corresponds to a jet that we also call $\gamma =\gamma(t)=(\gamma_0(t),\ldots,\gamma_g(t))$ with \[\gamma_i(t)= \sum\limits_{l=0}^mx_i^{(l)}t^l,\] where $x_i^{(l)}$ are the coordinates of $\gamma.$ The condition that $\gamma \in D_{m,k}$ is equivalent to the conditions $x_0^{(0)}=\cdots=x_0^{(kn_1-1)}=0,~x_0^{(kn_1)}\not=0$, and  $ord_t(f_s(\gamma))\geq m+1$ for $s=1,\ldots,g.$ From a little argument using Corollary~\ref{CorStructureBasis}, one can see that this implies the equalities $x_1^{(0)}=\cdots=x_1^{(kn_0-1)}= 0$ and ${x_1^{(kn_0)}}^{n_1}-{x_0^{(kn_1)}}^{n_0}=0.$ Since $x_0^{(kn_1)}\not=0,$ the last equation tells us that $x_1^{(kn_0)}\not=0.$ \\

Let us first examine the condition $ord_t(f_1(\gamma))\geq m+1.$ We have 
\begin{align*}
	f_1(\gamma)  \equiv f_1\bigg(\sum\limits_{l=kn_1}^mx_0^{(l)}t^l,\sum\limits_{l=kn_0}^mx_1^{(l)}t^l\bigg)& \equiv t^{kn_0n_1}f_1\bigg(\sum\limits_{l=kn_1}^mx_0^{(l)}t^{l-kn_1},\sum\limits_{l=kn_0}^m x_1^{(l)}t^{l-kn_0}\bigg)\\
	& \equiv t^{kn_0n_1}\sum_{l=0}^{m-kn_0n_1}\mathcal{F}_1^{(l)}t^l ~~\text{mod}~~(t^{m+1}),
\end{align*} 
where $\mathcal{F}_1^{(l)}:= F_1^{(l)}\big(x_0^{(kn_1)},x_1^{(kn_0)},\ldots,x_0^{(kn_1+l)},x_1^{(kn_0+l)}\big).$ The second equality follows from the weighted-homogeneity of $f_1$ with weights $n_1$ and $n_0$ for $x_0$ and $x_1$, respectively. Hence, the condition that $ord_t(f_1(\gamma))\geq m+1$ is equivalent to the annihilation of $\mathcal{F}_1^{(l)}$ for ${l=0,\ldots,m-kn_0n_1}$ with the condition that $x_0^{(kn_1)}\not=0.$ These are the defining equations of the jet schemes of the regular part of the curve defined by $x_1^{{(kn_0)}^{n_1}}-x_0^{{(kn_1)}^{n_0}}=0$. Clearly, $\mathcal F_1^0 = x_1^{{(kn_0)}^{n_1}}-x_0^{{(kn_1)}^{n_0}}$, and for $l = 1, \ldots, m-kn_0n_1$, they are of the form \[\mathcal{F}_1^{(l)}=\alpha_l {x_1^{(kn_0)}}^{n_1-1}x_1^{(kn_0+l)}-H_l\Big(x_0^{(kn_1)},x_1^{(kn_0)},\ldots,x_0^{(kn_1+l)},x_1^{(kn_0+l-1)}\Big),\] where $\alpha_l \in \co \setminus \{0\}$ and $H_l$ is a polynomial. Because $x_1^{(kn_0)}\not=0$, we can divide each $\mathcal{F}_1^{(l)}$ for $l = 1, \ldots, m - kn_0n_1$ by ${x_1^{(kn_0)}}^{n_1-1}$ to see that $\mathcal{F}_1^{(l)}$ is linear in $x_1^{(kn_0+l)}$, and that the expression of $x_1^{(kn_0+l)}$ does not depend on the variables $x_1^{(h)}$ for $h > kn_0+l.$ In other words, the system of equations given by $\mathcal{F}_1^{(l)}$ for $l=0,\ldots,m-kn_0n_1$ is triangular in the variables $x_1^{(kn_0)},\ldots, x_1^{(kn_0+m-kn_0n_1)}.$  \\

Let us now examine the conditions $ord_t(f_s(\gamma))\geq m+1$ for $s=2,\ldots,g$. We already know that $ord_t(\gamma_0)=kn_1=\frac{k\bar{\beta}_0}{e_1}$ and $ord_t(\gamma_1)=kn_0=\frac{k\bar{\beta}_1}{e_1}.$ Therefore, \[ord_t(x_0^{b_{s0}}\cdots x_{s-1}^{b_{s(s-1)}}(\gamma))\geq ord_t(\gamma_0^{b_{s0}}\gamma_1^{b_{s1}})=b_{s0}\frac{k\bar{\beta}_0}{e_1}+b_{s1}\frac{k\bar{\beta}_1}{e_1}\geq \frac{kn_2\bar{\beta}_2}{e_1}\geq m+1,\] where the last two inequalities follow from Lemma~\ref{LemmaDmk} and  our assumption on $m$, respectively. Since $f_s=x_s^{n_s}-x_0^{b_{s0}}\cdots x_{s-1}^{b_{s(s-1)}},$ it follows that $ord_t(f_s(\gamma))\geq m+1$ is equivalent to $ord_t(x_s^{n_s}(\gamma)) = ord_t(\gamma_s^{n_s})\geq m+1$, which is in turn equivalent to $x_s^{(0)} = \cdots = x_s^{([\frac{m}{n_s}])} = 0$. \\

To recapitulate, the embedding of $D_{m,k}$ in $\text{Spec} ~\co[\underline{x}^{(j)};j=0,\ldots,m]_{x_0^{(kn_1)}}$ is defined by the ideal \[\Big(x_r^{(0)},\ldots,x_r^{(\frac{kn_0n_1}{n_r}-1)},~\mathcal{F}_1^{(l)},~ x_s^{(0)},\ldots,x_s^{([\frac{m}{n_s}])};r = 0,1;~l=0,\ldots,m-kn_0n_1;~s=2,\ldots,g\Big),\] which is exactly the same as we claimed. The codimension of $D_{m,k}$ is equal to $c_{1,k}(m);$ indeed, the above ideal is a complete intersection because the system of equations given by $\mathcal{F}_1^{(l)}$ for $l=0,\ldots,m-kn_0n_1$ is triangular in the variables $x_1^{(kn_0)},\ldots,$ $x_1^{(kn_0+m-kn_0n_1)}.$ As $\mathcal F_1^0 = x_1^{{(kn_0)}^{n_1}}-x_0^{{(kn_1)}^{n_0}}$, we clearly also have \[D_{m,k}\simeq (Y^1\setminus \{0\})\times \co^{(g+1)(m+1)-c_{i,k}(m)-1}.\] 
 
We now proceed in the induction and assume that the description of the ideal which is given at the beginning of this proof is true for $i-1.$ We have to treat two cases.\\

\textbf{The case where $i < j(k)$.} We need to prove that the description for the ideal of $D_{m,k}$ is also true if $\frac{kn_i\lbeta_i}{e_1} \leq m < \frac{kn_{i+1}\lbeta_{i+1}}{e_1}$. Let $\gamma \in D_{m,k}$ be identified with its corresponding jet $\gamma(t)$. Because \[\pi_{m,kn_0n_1} = \pi_{m,\frac{kn_i\lbeta_i}{e_1}-1} \circ \pi_{\frac{kn_i\lbeta_i}{e_1}-1, kn_0n_1},\] we have \[D_{m,k} = \pi^{-1}_{m,\frac{kn_i\lbeta_i}{e_1}-1}\Big(D_{\frac{kn_i\lbeta_i}{e_1}-1,k}\Big).\] From the induction hypothesis, it follows that the coordinates of $\gamma$ satisfy, among others, the equations \[x_r^{(0)} = \cdots = x_r^{(\frac{k\bar{\beta}_r} {e_1}-1)} ={x_h^{(\frac{k\bar{\beta}_h}{e_1})}}^{n_h}-{x_0^{(\frac{k\bar{\beta}_0}{e_1})}}^{b_{h0}}\cdots {x_{h-1}^{(\frac{k\bar{\beta}_{h-1}}{e_1})}}^{b_{h(h-1)}} = 0;\] \[r = 0, \ldots, i-1;~ h=1,\ldots,i-1.\] Since $x_0^{(\frac{k\lbeta_0}{e_1})} \neq 0$, the equation for $h=1$ gives us that $x_1^{(\frac{k\bar{\beta}_1}{e_1})}\not=0.$ Then, the equation for $h = 2$ gives that $x_2^{(\frac{k\bar{\beta}_2}{e_1})}\not=0.$ We can repeat this to conclude that \[x_r^{(\frac{k\bar{\beta}_r}{e_1})}\not=0;~~~r=0,\ldots,i-1.\] Together with the other equations, this implies that $ord_t(\gamma_r)=\frac{k\bar{\beta}_r}{e_1}$ for $r=0,\ldots,i-1$. The induction hypothesis also tells us that $ord_t(\gamma_i) \geq \frac{k\lbeta_i}{e_1}$. We now investigate, modulo the defining ideal of $D_{\frac{kn_i\bar{\beta}_i}{e_1}-1,k},$ the condition $ord_t(f_i(\gamma))\geq m+1.$ \\

Note that the equation $f_i$ is weighted homogeneous of degree $\frac{n_i\lbeta_i}{e_1}$ if we give $x_r$ the weight $\frac{\bar{\beta}_r}{e_1}$ for $r=0,\ldots,i$. Therefore, \[f_i(\gamma)\equiv f_i\bigg(\sum\limits_{l=\frac{k\bar{\beta}_0}{e_1}}^mx_0^{(l)}t^l,\sum\limits_{l=\frac{k\bar{\beta}_1}{e_1}}^mx_1^{(l)}t^l,\ldots,\sum\limits_{l=\frac{k\bar{\beta}_i}{e_1}}^mx_i^{(l)}t^l\bigg) \equiv t^{\frac{kn_i\bar{\beta}_i}{e_1}}\sum_{l=0}^{m-\frac{kn_i\lbeta_i}{e_1}}\mathcal{F}_i^{(l)}t^l ~~~\text{mod}~~(t^{m+1}),\]
 where $\mathcal{F}_i^{(l)}:= F_i^{(l)}\Big(x_0^{(\frac{k\bar{\beta}_0}{e_1})},\ldots,x_i^{(\frac{k\bar{\beta}_i}{e_1})},\ldots, x_0^{(\frac{k\bar{\beta}_0}{e_1}+l)},\ldots,x_i^{(\frac{k\bar{\beta}_i}{e_1}+l)}\Big)$. More precisely, \[\mathcal{F}_i^{(0)}= {x_i^{(\frac{k\bar{\beta}_i}{e_1})}}^{n_i}-{x_0^{(\frac{k\bar{\beta}_0}{e_1})}}^{b_{i0}}\cdots {x_{i-1}^{(\frac{k\bar{\beta}_{i-1}}{e_1})}}^{b_{i(i-1)}},\] and, for $l=1,\ldots,m-\frac{kn_i\bar{\beta}_i}{e_1},$ \[\mathcal{F}_i^{(l)}= \alpha_l{x_i^{(\frac{k\bar{\beta}_i}{e_1})}}^{n_i-1}x_i^{(\frac{k\bar{\beta}_i}{e_1}+l)}-H_l\Big(x_0^{(\frac{k\bar{\beta}_0}{e_1})},\ldots, x_i^{(\frac{k\bar{\beta}_i}{e_1})},\ldots,x_0^{(\frac{k\bar{\beta}_0}{e_1}+l)},\ldots,x_{i-1}^{(\frac{k\bar{\beta}_{i-1}}{e_1}+l)}, x_i^{(\frac{k\bar{\beta}_i}{e_1}+l-1)}\Big)\] for some $\alpha_l \in \co\setminus \{0\}$ and a polynomial $H_l$. The condition $ord_t(f_i(\gamma))\geq m+1$ is thus given by the annihilation of $\mathcal{F}_i^{(l)}$ for $l=0,\ldots, m-\frac{kn_i\bar{\beta}_i}{e_1}.$ Because $x_r^{(\frac{k\bar{\beta}_r}{e_1})}\not=0$ for $r=0,\ldots,i-1$, the equation $\mathcal{F}_i^{(0)}=0$ gives that $ x_i^{(\frac{k\bar{\beta}_i}{e_1})}\not=0.$ Dividing ${x_i^{(\frac{k\bar{\beta}_i}{e_1})}}^{n_i-1}$ in $\mathcal{F}_i^{(l)}$ for $l\geq 1$, we again see that the system of equations is triangular in the variables $x_i^{(\frac{k\bar{\beta}_i}{e_1})}, \ldots, x_i^{(\frac{k\bar{\beta}_i}{e_1} + m - \frac{kn_i\bar{\beta}_i}{e_1})}$, and both the description of the ideal and the statement of the proposition follow.\\
 
\textbf{The case where $i = j(k)$.} 
In this case, it is enough to prove that $D_{\frac{kn_i\bar{\beta_i}}{e_1},k}=\emptyset.$ Suppose there exists an element $\gamma \in D_{\frac{kn_i\bar{\beta_i}}{e_1},k}$ and identify $\gamma$ once more with its jet $\gamma(t)$. On the one hand, like in the previous case, the induction hypothesis implies that \[x_r^{(\frac{k\bar{\beta}_r}{e_1})}\not=0;~~~r=0,\ldots,i-1.\] Therefore, \[ord_t(x_0^{b_{i0}}\cdots x_{i-1}^{b_{i(i-1)}}(\gamma))=\frac{kn_i\bar{\beta}_i}{e_1}.\] On the other hand, from the assumption $i = j(k)$ and the fact that $e_i =n_{i+1}\cdots n_g$, we have $\frac{kn_i\bar{\beta}_i}{e_1}= \frac{k'\bar{\beta}_i}{e_i}$ for some $k' \geq 1$ which is not a multiple of $n_i$. Since $n_i=\frac{e_{i-1}}{e_i}$ and  $e_i=gcd(e_{i-1},\bar{\beta}_i)$, we also know that $n_i$ and $\frac{\bar{\beta}_i}{e_i}$ are coprime. It follows that $n_i$ does not divide $\frac{kn_i\bar{\beta}_i}{e_1}=\frac{k'\bar{\beta}_i}{e_i}.$ As $f_i=x_i^{n_i}-x_0^{b_{i0}}\cdots x_{i-1}^{b_{i(i-1)}}$ and $ord_t(\gamma_i) \geq \frac{k\lbeta_i}{e_1}$ by the induction hypothesis, we can conclude that $ord_t(f_i(\gamma))=\frac{kn_i\bar{\beta}_i}{e_1}$. This contradicts that $\gamma \in D_{\frac{kn_i\bar{\beta}_i}{e_1},k}$. In other words, $D_{\frac{kn_i\bar{\beta}_i}{e_1},k}=\emptyset.$
   \end{proof}

We are now able to give the decomposition of $\pi_m^{-1}(0)_{red}$ into irreducible components.
   
\begin{theo}\label{ThmIrrCom}
Consider $m\geq 1$. Let $l\in \n$ be such that $ln_0n_1<m\leq(l+1)n_0n_1$ and let 
\begin{alignat*}{2}
	&D_{m,k} &&:= \pi_{m,kn_0n_1}^{-1}(\{x_0^{(0)}=x_0^{(1)}=\cdots=x_0^{(kn_1-1)}=0\}\cap \{x_0^{(kn_1)}\not=0\})_{red}, \quad k = 1, \ldots, l, \\
	&B_m &&:= \pi_{m,ln_0n_1}^{-1}(\{x_0^{(0)}=x_0^{(1)}=\cdots=x_0^{(ln_1)}=0\})_{red}.
\end{alignat*} 
The irreducible components of $\pi_{m}^{-1}(0)_{red}$ are $\C_{m,k} := \overline{D_{m,k}}$ for $k=1,\ldots,l$ such that ${m < \frac{kn_{j(k)}\bar{\beta}_{j(k)}}{e_1}}$ and $B_m.$ Furthermore, $B_m$ is a component of maximal dimension. 
\end{theo}

\begin{proof} 
If $l = 0$, then $\pi_{m}^{-1}(0)_{red} = B_m$ is irreducible, and there is nothing to prove. For $l \geq 1$, the stratification~\eqref{Stratification} and Proposition~\ref{PropDmk} tell us that \[\pi_{m}^{-1}(0)_{red}=\Big(\bigcup_{k=1}^l\C_{m,k}\Big) \cup B_m\] is a decomposition in closed irreducible subvarieties, and that the extra condition on $k$ comes from the fact that $\C_{m,k}=\emptyset$ for $m \geq \frac{k n_{j(k)}\bar{\beta}_{j(k)}}{e_1}$. We still need to prove that there are no inclusions between the closed sets in this union. We already have the following two non-inclusions (assuming that $\C_{m,k}, \C_{m,k'} \neq \emptyset$): 
\begin{enumerate}
	\item $\C_{m,k} \not\subset \C_{m,k'}$ for $k<k'\leq l,$ because $\C_{m,k'}\subset \{x_0^{(kn_1)}=0\}$ but $\C_{m,k}\not\subset \{x_0^{(kn_1)}=0\}$ by the definition of these closed sets; and
	\item $\C_{m,k} \not\subset B_m$ for all $k\leq l$ because $B_m\subset \{x_0^{(ln_1)}=0\}$ but $\C_{m,k}\not\subset \{x_0^{(ln_1)}=0\},$ again by definition of these sets.
\end{enumerate}
It remains to show that there are no inclusions in the other directions. This will follow from the following inequalities in codimensions, considered in $\co^{(g+1)(m+1)}$ (again assuming that $\C_{m,k}, \C_{m,k'} \neq \emptyset$):
\begin{enumerate}
	\item $\text{codim}(\C_{m,k})\geq\text{codim}(\C_{m,k'})$ for $k<k'\leq l$, and
	\item $\text{codim}(\C_{m,k})\geq \text{codim}(B_m)$ for all $k\leq l$.
\end{enumerate}
Indeed, the above non-inclusions tell us in particular that all these closed sets are not equal. Because they are also irreducible, the inequalities in codimensions imply that there are no inclusions in the other directions. \\  

We begin with remarking that the inequalities $\lbeta_2 > n_1\lbeta_1$ and $n_2 \geq 2$ imply that 
\begin{equation} \label{Ineq}
	\frac{hn_2\lbeta_2}{e_1} > (h+1)n_0n_1
\end{equation}
for all $h \geq 1$. In particular, we have $ln_0n_1 < m \leq (l+1)n_0n_1 < \frac{ln_2\lbeta_2}{e_1}$ such that $\C_{m,l} \neq \emptyset$ with $\text{codim}(\C_{m,l}) = c_{1,l}(m)$. One can now verify with the formulas of the codimension that $\text{codim}(\C_{m,l}) \geq \text{codim}(B_m)$, but this can also been seen from the following short argument. First, it is easy to check that $\C_{ln_0n_1+1,l}$ and $B_{ln_0n_1+1}$ have the same codimension (note that $\C_{ln_0n_1+1,l}\neq \emptyset$ of codimension $c_{1,l}(ln_0n_1 + 1)$). Second, for $n\in \n$ satisfying $ln_0n_1+2 \leq n \leq (l+1)n_0n_1 < \frac{ln_2\bar{\beta}_2}{e_1}$, it follows from Corollary~\ref{CorStructureBasis} that the equation $F_1^{(n)}$ contributes to the codimension of $B_n$ if and only if $n_0$ or $n_1$ divides $n$ (i.e., there is an extra variable $x_0^{(\frac{n}{n_0})}$ or $x_1^{(\frac{n}{n_1})}$ equal to $0$), while from the proof of Proposition~\ref{PropDmk}, $F_1^{(n)}$ always contributes to the codimension of $\C_{n,l}$. Additionally, the equations $F_j^{(n)}$ for $j=2,\ldots,g$ contribute to the codimension of $B_n$ if and only if they contribute to the codimension of $\C_{n,l}.$ Therefore, we have the inequality $\text{codim}(\C_{n,l})\geq \text{codim}(B_n)$ for $ln_0n_1 < n \leq (l+1)n_0n_1 < \frac{ln_2\bar{\beta}_2}{e_1}$, and in particular for $n=m$. \\

If $l=1$, we are done; we assume from now on that $l > 1$. For $k = 1, \ldots, l$, we define \[c_{k}(m) := k(n_0+n_1)+\sum_{l=2}^i\frac{k\bar{\beta}_l}{e_1}+ \sum_{l=1}^i\Big(m-\frac{kn_l\bar{\beta}_l}{e_1}+1\Big)+\sum_{l=i+1}^g\Big(\Big[\frac{m}{n_l}\Big]+1\Big),\] where $i \in \{1 , \ldots, g\}$ such that $\frac{kn_i\beta_i}{e_1} \leq m < \frac{kn_{i+1}\lbeta_{i+1}}{e_1}$. We can always find such a unique integer $i$ because $\lbeta_{g+1} = +\infty$ by convention. Furthermore, we know by Proposition~\ref{PropDmk} that $\text{codim}(\C_{m,k}) = c_{i,k}(m) = c_k(m)$ if $\C_{m,k} \neq \emptyset$ or ,equivalently, if $i < j(k)$. Consider now a fixed $k \in \{2, \ldots, l\}$. We will show that \[c_{k-1}(m) \geq c_k(m).\] This leads to the series of inequalities \[c_1(m) \geq c_2(m)\geq \cdots \geq c_l(m) = \text{codim}(\C_{m,l}) \geq \text{codim}(B_m),\] from which the above-mentioned inequalities in the codimension follow, and hence, the non-inclusions that we wanted to prove. In particular, our proof implies that $B_m$ is a component of maximal dimension, being of smallest codimension. \\

We first compare $c_{k-1}(kn_0n_1+1)$ and $c_k(kn_0n_1+1)$. From inequality~\eqref{Ineq} for $h = k-1$, it follows that $i = 1$ in $c_{k-1}(kn_0n_1)$, and that $i = 1$ in $c_{k-1}(kn_0n_1+1)$ if $kn_0n_1 + 1 < \frac{(k-1)n_2\lbeta_2}{e_1}$ or $i=2$ in $c_{k-1}(kn_0n_1+1)$ if $kn_0n_1 + 1 = \frac{(k-1)n_2\lbeta_2}{e_1}$. Using this, one can check that \[c_{k-1}(kn_0n_1+1) - c_{k-1}(kn_0n_1) \geq \text{codim}(B_{kn_0n_1+1}) - \text{codim}(B_{kn_0n_1}).\] With a same reasoning as before, one can see that $c_{k-1}(kn_0n_1) = \text{codim}(\C_{kn_0n_1,k-1})\geq\text{codim}(B_{kn_0n_1}).$ Together with $c_{k}(kn_0n_1+1) = \text{codim}(\C_{kn_0n_1+1,k}) = \text{codim}(B_{kn_0n_1+1})$, we obtain that \[c_{k-1}(kn_0n_1+1) \geq  c_k(kn_0n_1+1).\] Now, note that 
\begin{enumerate}
	\item the value of $c_k(n)$ increases when $n$ varies in the interval $\Big[\frac{kn_i\bar{\beta}_i}{e_1},\frac{kn_{i+1}\bar{\beta}_{i+1}}{e_1}\Big)\cap \mathbb{N}$ and it grows faster when $n$ varies in $\Big[\frac{kn_i\bar{\beta}_i}{e_1},\frac{kn_{i+1}\bar{\beta}_{i+1}}{e_1}\Big)\cap \mathbb{N}$ for greater $i$;
	\item the growing of $c_{k-1}(n)$ and $c_k(n)$ is completely the same when $n$ varies in the interval $\Big[\frac{(k-1)n_i\bar{\beta}_i}{e_1},\frac{(k-1)n_{i+1}\bar{\beta}_{i+1}}{e_1}\Big)\cap \mathbb{N}$ and $\Big[\frac{kn_i\bar{\beta}_i}{e_1},\frac{kn_{i+1}\bar{\beta}_{i+1}}{e_1}\Big)\cap \mathbb{N}$, respectively;
	\item the length of the interval $\Big[\frac{(k-1)n_i\bar{\beta}_i}{e_1},\frac{(k-1)n_{i+1}\bar{\beta}_{i+1}}{e_1}\Big)\cap \mathbb{N}$ is smaller than the length of $\Big[\frac{kn_i\bar{\beta}_i}{e_1},\frac{kn_{i+1}\bar{\beta}_{i+1}}{e_1}\big)\cap \mathbb{N}$.  
\end{enumerate}
 For these reasons, $c_{k-1}(n)$ grows faster with $n\geq kn_0n_1+1$ than $c_k(n).$ As $c_{k-1}(kn_0n_1+1) \geq  c_k(kn_0n_1+1)$, we can indeed deduce that $c_{k-1}(m) \geq c_k(m).$
\end{proof}

By comparing with Corollary 4.4 in~\cite{M1}, one can see that this theorem is a generalization of the case for $g=1$. More precisely, if $g = 1$, we can also stratify $\pi_m^{-1}(0)_{red}$ for all $m\in \n$ with $l \in \n$ such that $ln_0n_1 < m \leq (l+1)n_0n_1$ by 
\begin{equation} \label{StratificationForg=1}
 	\pi^{-1}_m(0)_{red} = \Big(\bigsqcup_{k=1}^l D_{m,k}\Big) \sqcup B_m,
\end{equation} where 
\begin{alignat*}{2}
	&D_{m,k} &&= \pi_{m,kn_0n_1}^{-1}(\{x_0^{(0)}=x_0^{(1)}=\cdots=x_0^{(kn_1-1)}=0\}\cap \{x_0^{(kn_1)}\not=0\})_{red}, \\
	&B_m &&= \pi_{m,ln_0n_1}^{-1}(\{x_0^{(0)}=x_0^{(1)}=\cdots=x_0^{(ln_1)}=0\})_{red}.
\end{alignat*}
The irreducible components are $C_{m,k} := \overline{D_{m,k}}$ for all $k = 1, \ldots, l$ and $B_m$. There is no extra condition on $k$ as all $D_{m,k}$ are non-empty. In other words, we can define $j(k) := 2$ for all $k \geq 1$. It is also easy to check that the codimension of $C_{m,k}$ is given by the formula of $c_{1,k}(m)$ in~\eqref{cik(m)} and that the codimension of $B_m$ is the same as in Corollary~\ref{CorDimBm}. In particular, $B_m$ is still a component of maximal dimension if $g = 1$. \\

As an easy corollary, we find the decomposition into irreducible components of $Y_m$ for all $g\geq 1$ and $m \geq 1$.

\begin{cor}\label{CorIrrComWhole}
Consider a space monomial curve $Y \subset \co^{g+1}$ defined by the equations~\eqref{EquationsY}. Consider $m\geq 1$, let $l\in \n$ be such that $ln_0n_1<m\leq(l+1)n_0n_1$ and let 
\begin{alignat*}{2}
	&D_{m,k} &&:= \pi_{m,kn_0n_1}^{-1}(\{x_0^{(0)}=x_0^{(1)}=\cdots=x_0^{(kn_1-1)}=0\}\cap \{x_0^{(kn_1)}\not=0\})_{red}, \quad k = 1, \ldots, l, \\
	&B_m &&:= \pi_{m,ln_0n_1}^{-1}(\{x_0^{(0)}=x_0^{(1)}=\cdots=x_0^{(ln_1)}=0\})_{red}.
\end{alignat*} 
The irreducible components of $Y_m$ are $\overline{\pi_m^{-1}(Y\setminus \{0\})}$, $\C_{m,k} := \overline{D_{m,k}}$ for $k=1,\ldots,l$ such that ${m < \frac{kn_{j(k)}\bar{\beta}_{j(k)}}{e_1}}$ and $B_m$. Furthermore, $B_m$ is a component of maximal dimension.
\end{cor}

\begin{proof}
Because $Y \setminus\{0\} \simeq \co \setminus\{0\}$ and $\pi_m$ restricted to $\pi^{-1}_m(Y \setminus \{0\})$ is a trivial fibration with fiber isomorphic to $\co^m$, we know that $\overline{\pi_m^{-1}(Y\setminus \{0\})}$ is irreducible of codimension $g(m+1)$ in $\co^{(g+1)(m+1)}$. As $\overline{\pi_m^{-1}(Y\setminus \{0\})}$ is trivially not contained in any of the components of $\pi_m^{-1}(0)$, it is now enough to prove the following inequalities in codimensions, considered in $\co^{(g+1)(m+1)}$ (assuming that $C_{m,k} \neq \emptyset$):
\begin{enumerate}
	\item $\text{codim}(\C_{m,k})\leq g(m+1)$ for $k\leq l$, and
	\item $\text{codim}(B_m)\leq g(m+1)$.
\end{enumerate}
In fact, we will show that strict inequality always holds in (1), while equality in (2) is possible if $g = 1$. \\

Recall from Proposition~\ref{PropDmk} that the codimension of $C_{m,k}$ (if non-empty) is given by $c_{i,k}(m)$ for $i \in \{1,\ldots, j(k)-1\}$ such that $\frac{kn_i\bar{\beta}_i}{e_1}\leq m<\frac{kn_{i+1}\bar{\beta}_{i+1}}{e_1}$. Because $n_l \geq 2$, this can be bounded from above by 
\begin{align*}
 c_{i,k}(m) &\leq k(n_0+n_1)+\sum_{l=2}^i\frac{k\bar{\beta}_l}{e_1}+ \sum_{l=1}^i\Big(m-\frac{kn_l\bar{\beta}_l}{e_1}+1\Big)+ (g-i)\Big(\frac{m}{2} + 1\Big) \\
 &= \frac{k}{e_1}\Big(\lbeta_0 - \sum_{l=1}^i(n_l-1)\lbeta_l \Big) + (g+i)\frac{m}{2} + g,
\end{align*}
where we also used that $n_0 = \frac{\lbeta_1}{e_1}$ and $n_1 = \frac{\lbeta_0}{e_1}$. Because $n_0>1$ and $n_1 > 1$ are coprime, we know that \[\lbeta_0 - (n_1-1)\lbeta_1  =  n_1\lbeta_1\Big(\frac{1}{n_0}  + \frac{1}{n_1} - 1\Big) < 0.\] It follows that $c_{i,k}(m) < (g+i)\frac{m}{2} + g$, which implies that $c_{i,k}(m) < g(m+1)$ as $i \leq g$. This proves the strict inequalities in (1). \\

If there exists a non-empty $C_{m,k}$, then strict inequality in (2) follows immediately from the fact that $\text{codim}(C_{m,k}) \geq \text{codim}(B_m)$, which was shown in Theorem~\ref{ThmIrrCom}. Otherwise, recall the codimension of $B_m$ determined in Corollary~\ref{CorDimBm} and let us first consider the case where $g = 1$. Because every $C_{m,k}$ is non-empty for $k = 1,\ldots,l$ if $l > 0$, it remains to study the case where $l = 0$. Then, $0 < m \leq n_0n_1$ and one can, for example with an easy induction argument, show that \[2 + \Big[\frac{m}{n_0}\Big] + \Big[\frac{m}{n_1}\Big] \leq m + 1,\] which implies the (possibly non-strict) inequality in (2). To conclude the proof for $g \geq 2$, it suffices to show that \[g + 1 + \sum_{i = 0}^g\Big[\frac{m}{n_i}\Big] < g(m+1).\] Since $n_i \geq 2$, this is trivially true for $m = 1$. For $m \geq 2$, we can use that $n_i \geq 2$ and $n_0 > n_1$ to find that
\begin{equation}\label{IneqIntPart}
 	\sum_{i = 0}^g\Big[\frac{m}{n_i}\Big] < (g+1)\frac{m}{2},
 \end{equation}
from which it is easy to see that the above inequality indeed holds.
\end{proof}

\begin{rem}\label{RemDimLct}
Because $B_m$ is a component of maximal dimension (equivalently, of minimal codimension) of $Y_m$ for $m\geq 1$, we can apply Mustata's formula~\cite[Corollary 3.4]{Mu1} to obtain the \emph{log canonical threshold} of the pair $(\mathbb{C}^{g+1},Y)$: \[lct(Y,\mathbb{C}^{g+1})= \min_{m \geq 0}\frac{\text{codim}(Y_m,\mathbb{C}^{(g+1)(m+1)})}{m+1}=\min\left(g,\min_{m\geq 1}\frac{\text{codim}(B_m,\mathbb{C}^{(g+1)(m+1)})}{m+1}\right).\] We will obtain its value in Corollary~\ref{lct} from the well-known fact that $-lct(Y,\mathbb{C}^{g+1})$ is the largest pole of the topological Igusa zeta function, see for instance~\cite[Theorem 3.5]{NX}. 
\end{rem} 

In Section~\ref{SpaceCurve}, we have defined the family $\eta:(\chi,0) \rightarrow (\co,0)$ whose special fiber is (the germ of) $Y$ and whose generic fiber is a plane branch, and which is equisingular as all curves have the same semigroup. This (flat) family is also equisingular in the sense that we have a simultaneous embedded resolution of all its fibers, see~\cite[Theorem 6.1]{GT} and~\cite[Theorem 33]{LMR}. A third type of equisingularity criterion is to ask if this family induces a flat family on the level of jet schemes in the following way. Put $S := (\co,0)$. For every $m \in \n$, we can consider the \emph{relative $m$th jet scheme} $((\chi,0)/S)_m$ of $\eta:(\chi,0) \rightarrow S$. These kinds of jet schemes can be defined using Hasse-Schmidt derivations and generalize the jet schemes $X_m = (X/\co)_m$ that we introduced for a complex affine variety in Section~\ref{JetMotivic}. See for example~\cite{V} for an introduction to relative jet schemes. This yields a natural morphism $\eta_m:((\chi,0)/S)_m \rightarrow S$ whose fibers are isomorphic to the $m$th jet schemes of the fibers of $\eta$, see~\cite[Proposition 5.6]{V}. In other words, this induces a family on the level of the $m$th jet schemes and we can investigate whether this is a flat family. It follows from~\cite[Corollary 4.13]{M1} that this family (with the reduced structure on its fibers) is flat outside the special fiber for every $g \geq 1$ and $m \in \n$. We will show below that the whole family $\eta_m:((\chi,0)/S)_m \rightarrow S$ is not flat for $g\geq 2$ and most values of $m \in \n$. This can be compared with the result in~\cite[Theorem 3.4]{LA} which states that a family of hypersurfaces admitting a simultaneous embedded resolution of singularities does induce a flat family on the level of jet schemes (with their reduced structures) for every $m$: in our case, while the generic fibers are hypersurfaces, the embedding dimension of $Y$ is $g.$  

\begin{theo}\label{TheoremNotFlat} For $g \geq 3$ and $m\geq 1,$  the family $\eta_m:((\chi,0)/S)_m \rightarrow S$  is not flat. For $g=2$ and $m$ big enough, the family  $\eta_m:((\chi,0)/S)_m \rightarrow S$ is not flat.
\end{theo} 

\begin{proof}
For $g \geq 2$ and $m$ satisfying the conditions in the statement, we will prove that the dimension of the special fiber of $\eta_m:((\chi,0)/S)_m \rightarrow S$ is strictly larger than the dimension of its generic fiber by showing that the codimension of its generic fiber is strictly larger than the codimension of its special fiber, both considered in $\mathbb{C}^{(g+1)(m+1)}.$ The codimension of the special fiber in $\mathbb{C}^{(g+1)(m+1)}$ is equal to the codimension of $B_m$, which we have found in Corollary~\ref{CorDimBm}. The codimension of the generic fiber in $\mathbb{C}^{2(m+1)}$ is given in Corollary 4.10 from~\cite{M1}; it is equal to \[ 2 + \Big[\frac{m}{\bar{\beta}_0}\Big] +\Big[\frac{m}{\bar{\beta}_1}\Big] \text ~~\text{ or }~~ 1+ \Big[\frac{m}{\bar{\beta}_0}\Big] +\Big[\frac{m}{\bar{\beta}_1}\Big],\] depending on some conditions on $m$. Since the jet schemes are independent of the embedding, we can compute the codimension of the generic fiber in $\mathbb{C}^{(g+1)(m+1)}$ by using the fact that the dimension is the same in both embedding spaces. More precisely, we have \[(g+1)(m+1)-(\text{codimension in } \mathbb{C}^{(g+1)(m+1)})= 2(m+1)-(\text{codimension in }\mathbb{C}^{2(m+1)}).\] We distinguish three cases. \\

\textbf{The case $g\geq 4$.}
It is enough to prove for every $m \geq 1$ that \[(g-1)(m+1)+1+\Big[\frac{m}{\bar{\beta}_0}\Big]+\Big[\frac{m}{\bar{\beta}_1}\Big] >g+1+ \sum_{i=0}^g \Big[\frac{m}{n_i}\Big], \] which is equivalent to 
\begin{equation}\label{TP}
 	(g-1)m > 1 + \sum_{i=0}^g \Big[\frac{m}{n_i}\Big] -\Big[\frac{m}{\bar{\beta}_0}\Big]-\Big[\frac{m}{\bar{\beta}_1}\Big].
 \end{equation}
Since $n_i \geq 2$, this is clearly true if $m = 1$. If $m \geq 2$, the upper bound~\eqref{IneqIntPart} in the proof of Corollary~\ref{CorIrrComWhole} yields
\begin{equation}\label{bound}
 	(g+1)\frac{m}{2}> \sum_{i=0}^g\Big[\frac{m}{n_i}\Big]-\Big[\frac{m}{\bar{\beta}_0}\Big]-\Big[\frac{m}{\bar{\beta}_1}\Big].
\end{equation}
Therefore, it is sufficient to show that $(g-1)m \geq 1 + (g+1)\frac{m}{2},$ which is true for $g\geq 4$ and $m \geq 2$.\\

\textbf{The case $g = 3$.}
For $m \geq \lbeta_0$, we can again prove the inequality~\eqref{TP}; because $[\frac{m}{\bar{\beta}_0}] \geq 1$, this follows from decreasing the upper bound in~\eqref{bound} to $2m - 1$. For $m < \lbeta_0$, the inequality~\eqref{TP} is not true in general. However, the codimension in $\mathbb{C}^{2(m+1)}$ of the generic fiber in this case is given by \[2 + \Big[\frac{m}{\bar{\beta}_0}\Big] +\Big[\frac{m}{\bar{\beta}_1}\Big]\] so that it is enough to show that \[2m > \Big[\frac{m}{n_0}\Big]+\Big[\frac{m}{n_1}\Big]+\Big[\frac{m}{n_2}\Big]+\Big[\frac{m}{n_3}\Big] -\Big[\frac{m}{\bar{\beta}_0}\Big]-\Big[\frac{m}{\bar{\beta}_1}\Big],\] which is true by~\eqref{bound}.\\

\textbf{The case $g = 2$.} 
In this case, our claim does not hold in general; it is easy to find examples in which the (co)dimension of the generic fiber and the special fiber are equal for some small $m\geq 1$. However, we will prove that for $m$ big enough, the claim does always hold. We again consider the inequality~\eqref{TP}. By using that $[\frac{m}{n}] \leq \frac{m}{n} < [\frac{m}{n}] +1$ for any positive integer $n$, it is enough to investigate \[m \geq  3 + \frac{m}{n_0}+\frac{m}{n_1}+\frac{m}{n_2} -\frac{m}{\bar{\beta}_0}-\frac{m}{\bar{\beta}_1}.\] Because $\bar{\beta}_0=n_1n_2$ and $\bar{\beta}_1=n_0n_2$ if $g = 2$, we can rewrite this as \[n_0n_1n_2m > 3n_0n_1n_2 + (n_1n_2 + n_0n_2 + n_0n_1 - n_0 - n_1)m,\] which is equivalent to $m \geq \frac{3n_0n_1n_2}{(n_2-1)(n_0n_1-n_0-n_1)}$ (note that $n_0n_1 - n_0 - n_1 > 0$ as $n_0 > n_1 \geq 2$ are coprime). Hence, for $m$ satisfying this lower bound, the codimension of the generic fiber is certainly bigger than the codimension of the special fiber.
\end{proof}

\section{Motivic zeta function of space monomial curves with plane semigroups} \label{Motivic}
	
Using the results from the previous section, we are now able to compute the series \[J_Y(T) = \sum_{m \geq 0}[Y_m](\Lmb^{-(g+1)}T)^{m+1}\] and to deduce the motivic Igusa zeta function of a space monomial curve $Y \subset \co^{g+1}$. \\
	
Assume first that $g \geq 2$. We start with recalling the stratification~\eqref{Stratification} of $\pi^{-1}_m(0)_{red}$ for $m \in \n$ and $l \in \n$ such that $ln_0n_1 < m \leq (l+1)n_0n_1$ given by \[\pi^{-1}_m(0)_{red} = \Big(\bigsqcup_{k=1}^l D_{m,k}\Big) \sqcup B_m,\] where 
\begin{alignat*}{2}
	&D_{m,k} &&= \pi_{m,kn_0n_1}^{-1}(\{x_0^{(0)}=x_0^{(1)}=\cdots=x_0^{(kn_1-1)}=0\}\cap \{x_0^{(kn_1)}\not=0\})_{red}, \\
	&B_m &&= \pi_{m,ln_0n_1}^{-1}(\{x_0^{(0)}=x_0^{(1)}=\cdots=x_0^{(ln_1)}=0\})_{red}.
\end{alignat*}
If $ln_0n_1 < m < (l+1)n_0n_1$, Corollary~\ref{CorStructureBasis} implies that $B_m \simeq \co^{(g+1)(m+1) - c(m)}$ with $c(m) := g + 1 + \sum_{i=0}^g [\frac{m}{n_i}].$ For $m = (l+1)n_0n_1$, the same corollary gives us the defining equations of $B_m$, which is not isomorphic to $\co^{(g+1)(m+1) - c(m)}$, but to the product of an affine space and the hypersurface defined by $x_1^{((l+1)n_0))^{n_1}} - x_0^{((l+1)n_1))^{n_0}} = 0$. However, the singular part of $B_{(l+1)n_0n_1}$, \[B_{(l+1)n_0n_1} \cap \{x_0^{((l+1)n_1)} = 0\},\] is isomorphic to $\co^{(g+1)((l+1)n_0n_1 + 1) - c((l+1)n_0n_1)}.$ Furthermore, it is easy to see that the regular part, \[B_{(l+1)n_0n_1} \cap \{x_0^{((l+1)n_1)} \neq 0\},\] is equal to $D_{(l+1)n_0n_1, l+1}$. Hence, if we define \[\mathcal{B}_m := \left\{\begin{array}{ll} \{0\} & \text{ if } m=0 \\ B_m \cap  \{x_0^{(ln_1)} = 0\} & \text{ if } m  = ln_0n_1 \text{ for some } l > 0 \\ B_m & \text{ if } ln_0n_1 < m < (l+1)n_0n_1 \text{ for some } l \geq 0 ,\end{array} \right. \] then each $\mathcal{B}_m \simeq \co^{(g+1)(m+1) - c(m)}$, and we find for all $m, l \in \n$ with $ln_0n_1 \leq m < (l+1)n_0n_1$ that \[\pi^{-1}_m(0)_{red} = \Big(\bigsqcup_{k=1}^l D_{m,k}\Big) \sqcup \mathcal{B}_m.\] This is a stratification in locally closed subvarieties for which $D_{m,k}=\emptyset$ if $m>\frac{kn_{j(k)}\lbeta_{j(k)}}{e_1}$. \\

These new stratifications can be visualized with a tree as in Figure~\ref{GeneralPicture}. On the vertical axis, we collect $\mathcal{B}_m$ for all $m \in \n$, but to shorten the notation, we only write $m$ instead of $\mathcal{B}_m$. This axis will be referred to as the \emph{main axis} or \emph{main branch}. For each $k \geq 1$, we construct a side branch at $kn_0n_1$ consisting of $D_{m,k}$ for all $m$ such that $kn_0n_1 \leq m < \frac{kn_{j(k)}\bar{\beta}_{j(k)}}{e_1},$ where $\lbeta_{g+1} = +\infty$. We again use a shorter notation and call this the \emph{side branch associated with $k$}. If $j(k) < g+1$, this side branch stops at $D_{m,k}$ for $m = \frac{kn_{j(k)}\lbeta_{j(k)}}{e_1} - 1$; if $j(k) = g+1$ (i.e., $n_2\cdots n_g$ divides $k$), this side branch never stops, and the part starting from $\frac{kn_g\lbeta_g}{e_1}$ is called the \emph{infinite branch associated with} $k$. In such a general picture, it is hard to give the side branches the correct length, and the tree should be interpreted as if the decomposition of $\pi^{-1}_m(0)_{red}$ for some $m\in \n$ can be reconstructed by drawing a horizontal line starting from the main axis and taking all intersections with the side branches. \\

\begin{figure}[h!] 
	\includegraphics[width=0.85\textwidth]{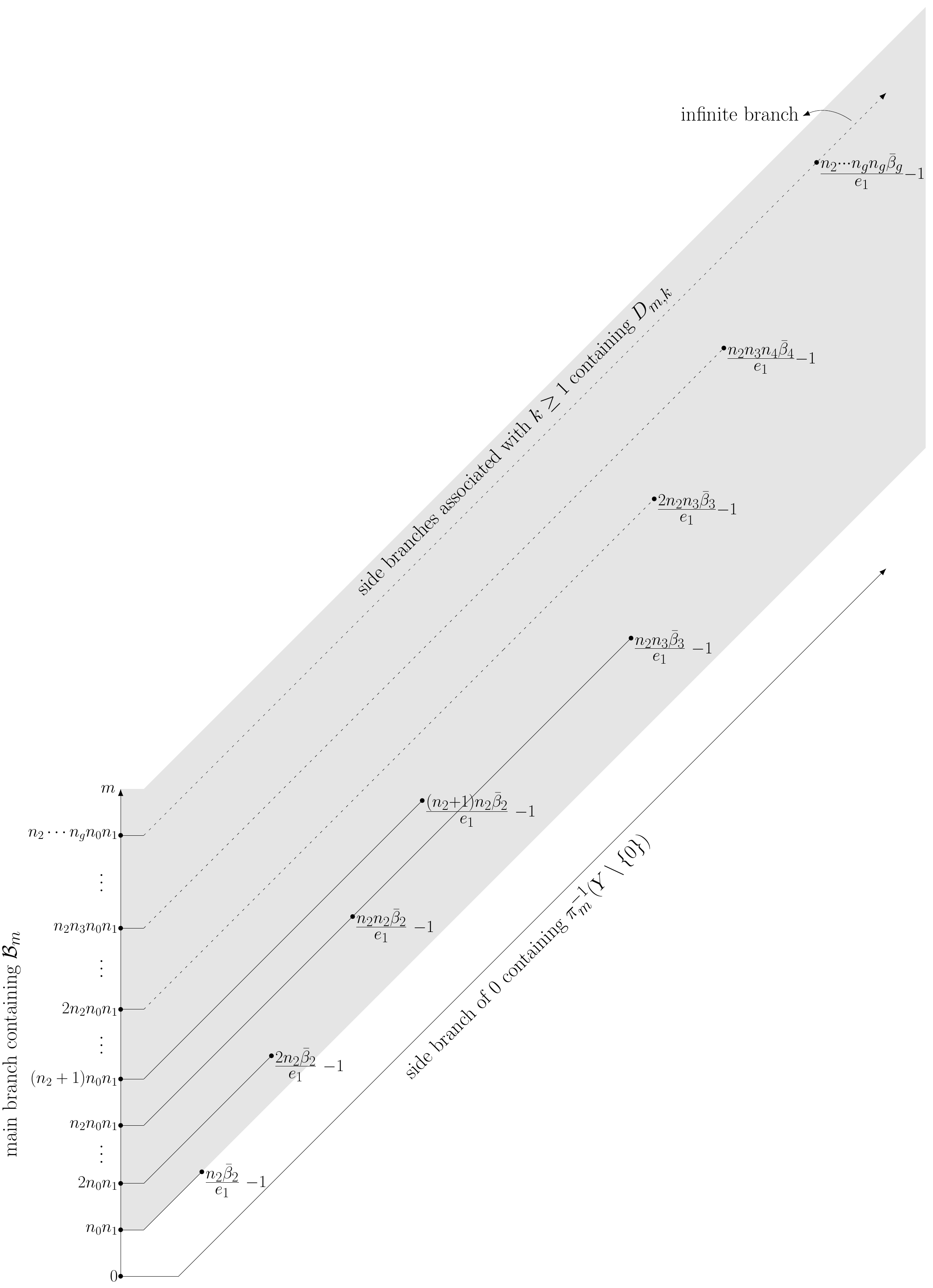}
	\caption{Visualization of the stratification of $Y_m$ for all $m \geq 0$.}
	\label{GeneralPicture}
\end{figure}

If we add a last side branch at $0$ containing $\pi^{-1}_m(Y \setminus \{0\})$ for every $m \in \n$, then the tree contains all information needed to compute $J_Y(T)$. More precisely, we have for every $m \in \n$ and $l \in \n$ satisfying $ln_0n_1 \leq m < (l+1)n_0n_1$ that \[[Y_m] =  [\pi^{-1}_m(Y \setminus \{0\})] + [\pi^{-1}_m(0)_{red}] = [\pi^{-1}_m(Y \setminus \{0\})] + [\mathcal{B}_m] + \sum_{k=1}^l[D_{m,k}].\] Hence, to sum $[Y_m]$ over all $m \geq 0$, we can first consider the side branch of $0$, the main branch, and the side branches for $k \geq 1$ separately, and then collect these totals to find the whole series $J_Y(T)$. \\

Let us first take a look at the side branch of $0$. Because $Y \setminus\{0\} \simeq \co \setminus\{0\}$ and $\pi_m$ induces a trivial fibration over $Y \setminus \{0\}$ with fiber $\co^m$, we have that $[\pi^{-1}_m(Y \setminus \{0\})] = (\Lmb-1)\Lmb^m$ by Remark~\ref{RemGrothRing}, and a simple calculation leads to the following expression. 	 
	
\begin{prop} \label{PropAboveReg}
	We have \[\sum_{m \geq 0}[\pi^{-1}_m(Y \setminus \{0\})](\Lmb^{-(g+1)}T)^{m+1} = \frac{(\Lmb-1)\Lmb^{-(g+1)}T}{1-\Lmb^{-g}T}.\]
\end{prop}
	
The computations for the main axis are also easy. Let $N_1 := \lcm(n_0,\ldots, n_g)$ be the least common multiple of $n_0, \ldots, n_g$ and put $\nu_1 :=  \sum_{l=0}^g \frac{N_1}{n_l}$.

\begin{prop} \label{PropMainBranch}
	 The contribution of the main branch to $J_Y(T)$ is \[\sum_{m \geq 0}[\mathcal{B}_m](\Lmb^{-(g+1)}T)^{m+1} = \frac{\Lmb^{-(g+1)}T}{1- \Lmb^{-\nu_1}T^{N_1}}\sum_{r=0}^{N_1-1}\Lmb^{-\sum\limits_{i=0}^g [\frac{r}{n_i}]}T^r.\]
\end{prop}
	
\begin{proof}
Because $\mathcal{B}_m \simeq \co^{(g+1)(m+1) - c(m)}$, we need to compute \[\sum\limits_{m\geq 0}\Lmb^{-c(m)}T^{m+1} = \Lmb^{-(g+1)}T\sum\limits_{m\geq 0}\Lmb^{-\tilde{c}(m)}T^m,\] where $c(m) = g+1 + \sum_{i=0}^g [\frac{m}{n_i}]$ and $\tilde{c}(m) = \sum_{i=0}^g [\frac{m}{n_i}]$. To this end, note that \[\nu_1 = \tilde{c}(m + N_1) - \tilde{c}(m)\] for all $m \in \n$. Hence, we can rewrite 	
\begin{align*}
	\sum\limits_{m\geq 0}\Lmb^{-\tilde{c}(m)}T^m 
	&= \sum_{r=0}^{N_1-1} \sum_{m\geq 0} \Lmb^{-\tilde{c}(mN_1 + r)}T^{mN_1 + r} \\
	&= \sum_{r=0}^{N_1-1} \sum_{m\geq 0} \Lmb^{-\left(m\nu_1 + \tilde{c}(r)\right)}T^{mN_1 + r} \\
	& = \frac{1}{1- \Lmb^{-\nu_1}T^{N_1}}\sum_{r=0}^{N_1-1}\Lmb^{-\tilde{c}(r)}T^r,
\end{align*}
which gives the desired expression.
\end{proof}

\begin{rem}
 In the proof of Proposition~\ref{PropMainBranch}, we found that $\nu_1 = \tilde{c}(m + N_1) - \tilde{c}(m)$ for all $m \in \n$ by looking for a positive integer $N$ such that $\tilde{c}(m)$ is \emph{linear on congruence classes modulo} $N$. That is, \[\tilde{c}(m  + N) = \tilde{c}(m) + \tilde{c}(N)\] for all $m \in \n$. In order to make $\tilde{c}(m) = \sum_{i=0}^g [\frac{m}{n_i}]$ linear on congruence classes modulo $N$ for any choice of $n_0, \ldots, n_g$, we need to impose that $n_i$ divides $N$ for all $i = 0, \ldots, g$. Clearly, $N = N_1 = \lcm(n_0, \ldots, n_g)$ is the smallest integer satisfying this condition. In fact, the \lq period\rq\ $N$ could be any common multiple of $n_0, \ldots, n_g$. This does not make any difference for the poles of the motivic zeta function because the ratio $\frac{\tilde{c}(N)}{N}$ stays the same. However, it is more natural to take the smallest period as this leads to the smallest remaining sum $\sum_{r = 0}^{N-1}\Lmb^{-\sum_{i=0}^g [\frac{r}{n_i}]}T^r$.
 \end{rem}
 	
The rest of this section will be mainly devoted to the contribution of the side branches associated with $k \geq 1$, which is by Proposition~\ref{PropDmk} given by \[\sum_{k \geq 1} \sum_{i=1}^{j(k)-1} \sum_{m=\frac{kn_i\lbeta_i}{e_1}}^{\frac{kn_{i+1}\lbeta_{i+1}}{e_1}-1}  [D_{m,k}](\Lmb^{-(g+1)}T)^{m+1} = \sum_{k \geq 1} \sum_{i=1}^{j(k)-1} \sum_{m=\frac{kn_i\lbeta_i}{e_1}}^{\frac{kn_{i+1}\lbeta_{i+1}}{e_1}-1} (\Lmb-1) \Lmb^{-(c_{i,k}(m) + 1)}T^{m+1},\] where $c_{i,k}(m)$ is defined in~\eqref{cik(m)}. Let us consider for a moment the part where $i=1$ for all $k \geq 1$: \[\sum_{k \geq 1} \sum_{m=kn_0n_1}^{\frac{kn_2\lbeta_2}{e_1}-1}  (\Lmb-1) \Lmb^{-(c_{1,k}(m) + 1)}T^{m+1}.\] Each interval $\big[kn_0n_1,\frac{kn_2\lbeta_2}{e_1}\big[\cap \n$ can be partitioned in $k$ intervals \[I_{1,k}^{(p)} := \bigg[kn_0n_1 + (p-1)\Big(\frac{n_2\lbeta_2}{e_1} - n_0n_1 \Big), kn_0n_1 + p\Big(\frac{n_2\lbeta_2}{e_1} - n_0n_1 \Big) \bigg[ \cap \n,\] $p=1, \ldots, k$, of the same length $l_1 := \frac{n_2\lbeta_2}{e_1} - n_0n_1$ as in Figure~\ref{IntervalsPicture}. Using these intervals, the part for $i=1$ can be rewritten as \[\sum_{p \geq 1}\sum_{\kappa \geq 0} \sum_{I_{1,p+\kappa}^{(p)}}(\Lmb-1)\Lmb^{-(c_{1,p+\kappa}(m)+1)}T^{m+1},\] where, from now on, a sum over all $m \in I$ for some interval $I \subset \n$ is written in a shorter way as $\sum_{I}$. We will first concentrate on computing the sum over all $\kappa$ and $m$ for $p$ fixed. In other words, we will first sum for each interval $I_{1,k}^{(p)}$ over all \emph{suitable} $k$, meaning that $I_{1,k}^{(p)}$ appears in the partition of $\big[kn_0n_1,\frac{kn_2\lbeta_2}{e_1}\big[\cap \n$. We will refer to this as \emph{vertical summation} inspired by Figure~\ref{IntervalsPicture}. Afterwards, we will sum these totals over all $p \geq 1$, called \emph{horizontal summation}.  \\
	
For $i = 2, \ldots, g-1$, we need to consider all $k \geq 1$ with $i < j(k)$ or, in other words, all multiples of $n_2\ldots n_i$. For each such $k = k' n_2 \cdots n_i$, we can partition the interval $\big[\frac{kn_i\lbeta_i}{e_1}, \frac{kn_{i+1}\lbeta_{i+1}}{e_1}\big[\cap \n$ in $k'$ intervals \[I_{i,k}^{(p)} := \bigg[\frac{kn_i\lbeta_i}{e_1} + (p-1)\Big(\frac{n_{i+1}\lbeta_{i+1}}{e_i} - \frac{n_i\lbeta_i}{e_i} \Big), \frac{kn_i\lbeta_i}{e_1} + p\Big(\frac{n_{i+1}\lbeta_{i+1}}{e_i} - \frac{n_i\lbeta_i}{e_i} \Big) \bigg[\cap \n,\] $p=1, \ldots, k'$, of length $l_i := n_2\cdots n_i\big(\frac{n_{i+1}\lbeta_{i+1}}{e_1} - \frac{n_i\lbeta_i}{e_1}\big) = \frac{n_{i+1}\lbeta_{i+1}}{e_i} - \frac{n_i\lbeta_i}{e_i}$. With this notation, the part for $i$ is equal to \[\sum_{p \geq 1}\sum_{\kappa \geq 0} \sum_{I_{i,(p+\kappa)n_2\cdots n_i}^{(p)}}(\Lmb-1)\Lmb^{-(c_{i,(p+\kappa)n_2\cdots n_i}(m)+1)}T^{m+1}.\] We will again first sum vertically (for fixed $p$) and then horizontally (over $p\geq1$). \\
	
\begin{figure}[h!] 
	\includegraphics[width=0.7\textwidth]{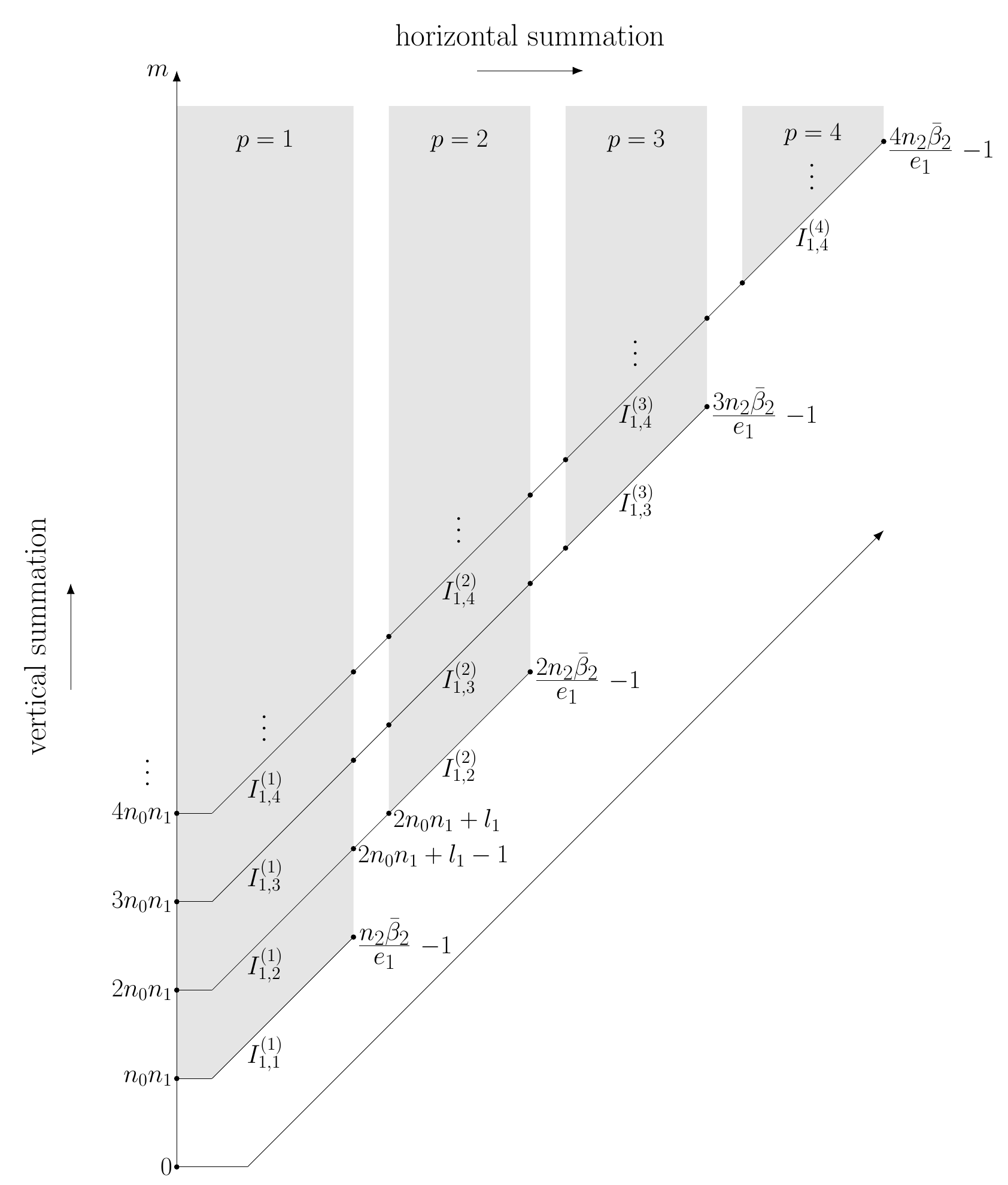}
	\caption{Visualization of the intervals $I_{1,k}^{(p)}$ for $k \geq 1$ and $p \leq k$.}
	\label{IntervalsPicture}
\end{figure}

For $i = g$, we study the infinite branches for $k = k' n_2 \cdots n_g$ and $j(k) = g+1$. With \[I_{g,k}^{(1)} := \bigg[\frac{kn_g\lbeta_g}{e_1}, + \infty \bigg[ \cap \n,\] the infinite branches lead to \[\sum_{\kappa \geq 0} \sum_{I_{g,(1+\kappa)n_2\cdots n_g}^{(1)}}(\Lmb-1)\Lmb^{-(c_{g,(1+\kappa)n_2\cdots n_g}(m)+1)}T^{m+1}.\] This consists of only one vertical sum. \\
	
To find the vertical summation, we start with a lemma. Recall that $N_1$ is the least common multiple of $n_0,\ldots,n_g$ and that $\nu_1 = \sum_{l=0}^g \frac{N_1}{n_l}$. We can generalize these numbers by introducing for $i = 1,\ldots, g$ the positive integers \[N_i := \lcm\Big(\frac{\lbeta_i}{e_i}, n_i,\ldots, n_g\Big),\] and \[\nu_i := N_i\left(\frac{1}{n_i\lbeta_i} \bigg(\sum_{l=0}^i\lbeta_l - \sum_{l=1}^{i-1}n_l\lbeta_l\bigg) + (i-1)+ \sum_{l=i+1}^g \frac{1}{n_l}\right);\] for $i=1$, we recover the earlier expressions since $n_0 = \frac{\lbeta_1}{e_1}$ and $n_1 = \frac{\lbeta_0}{e_1}$. Using the definition~\eqref{cik(m)} of $c_{i,k}(m)$, the next lemma is a straightforward calculation.
	
\begin{lem} \label{LemmaVerSum}
\begin{enumerate}[wide, labelindent=0pt]
	\item For all $k \geq 1$ and $0 \leq m < kl_1=k\big(\frac{n_2\lbeta_2}{e_1}-n_0n_1\big)$, we have \[\nu_1 = c_{1,k+\frac{N_1}{n_0n_1}}\Big(\big(k+\frac{N_1}{n_0n_1}\big)n_0n_1+m\Big) - c_{1,k}(k n_0n_1+m).\]
	\item For $i=2, \ldots, g$ and $k \geq 1$, we have for every $m\in \n$ lying in the interval $\big[0, k l_i[ = [0, k n_2\cdots n_i\big(\frac{n_{i+1}\lbeta_{i+1}}{e_1} - \frac{n_i\lbeta_i}{e_1}\big) \big[$ that \[\nu_i = c_{i,(k+\frac{e_iN_i}{n_i\lbeta_i})n_2\cdots n_i}\Big(\big(k+\frac{e_iN_i}{n_i\lbeta_i}\big)n_2 \cdots n_i\frac{n_i\lbeta_i}{e_1}+m\Big) - c_{i,k n_2\cdots n_i}\Big(k n_2 \cdots n_i\frac{n_i\lbeta_i}{e_1}+m\Big).\] 
\end{enumerate}
\end{lem}
	
Note that $K_1 := \frac{N_1}{n_0n_1}$ is an integer as $n_0$ and $n_1$ are coprime divisors of $N_1$. Furthermore, it is the smallest integer (for general $n_0, \ldots, n_g$) such that $K_1n_0n_1$ is divisible by $n_2,\ldots, n_g$ or, in other words, such that the sum $\sum_{l=2}^g[\frac{m}{n_l}]$ in $c_{1,k}(m)$ is linear on congruence classes modulo $K_1n_0n_1$. Similarly, every $K_i := \frac{e_iN_i}{n_i\lbeta_i}$ for $i=2, \ldots, g$ is the smallest integer (for general $n_0, \ldots, n_g$) making $\sum_{l=i+1}^g[\frac{m}{n_l}]$ linear modulo $K_in_2 \cdots n_i\frac{n_i\lbeta_i}{e_1} = K_i\frac{n_i\lbeta_i}{e_i}$. This idea was used in the proof of Proposition~\ref{PropMainBranch}, and we will continue following the approach of this proof to show the results in the next proposition. The first two results say that for each $i = 1, \ldots, g-1$ and $p \geq 1$, we only need to know the behavior on the first $K_i$ side branches associated with suitable $k$ (i.e., the partition of the side branch of $k$ contains $I_{i,k}^{(p)}$) in order to know the whole vertical summation. This motivates our choice for the smallest integers $K_i$, and it is again easy to check that this does not influence the ratios $\frac{\nu_i}{N_i}$. \\

\begin{prop} \label{PropVerSum}
\begin{enumerate} [wide, labelindent=0pt]
	\item For $i = 1$ and all $p \geq 1$, the vertical summation gives \[\frac{\Lmb-1}{1-\Lmb^{-\nu_1}T^{N_1}}\sum_{r=0}^{\frac{N_1}{n_0n_1}-1}\sum_{I_{1,p+r}^{(p)}}\Lmb^{-(c_{1,p+r}(m)+1)}T^{m+1}.\]
	\item For every $i = 2, \ldots, g-1$ and $p \geq 1$, the vertical summation gives \[\frac{\Lmb - 1}{1-\Lmb^{-\nu_i}T^{N_i}}\sum_{r=0}^{\frac{e_iN_i}{n_i\lbeta_i}-1} \sum_{I_{i,(p+r)n_2\cdots n_i}^{(p)}} \Lmb^{-(c_{i,(p+r)n_2\cdots n_i}(m) + 1)}T^{m+1}.\]
	\item For the infinite branches, we find \[\frac{(\Lmb-1)\Lmb^{-(\nu_g + g + 1)}T^{N_g + 1}}{(1-\Lmb^{-g}T)(1-\Lmb^{-\nu_g}T^{N_g})}.\]
\end{enumerate}
\end{prop}

\begin{proof}
For the first part, we need to consider \begin{align*}
	&\sum_{\kappa \geq 0} \sum_{I_{1,p+\kappa}^{(p)}}(\Lmb-1)\Lmb^{-(c_{1,p+\kappa}(m)+1)}T^{m+1} \\
	& = \sum_{\kappa \geq 0}\sum_{m = (p-1)l_1}^{pl_1-1}(\Lmb-1)\Lmb^{-(c_{1,p+\kappa}((p+\kappa)n_0n_1 + m)+1)}T^{(p+\kappa)n_0n_1 + m + 1}
\end{align*}
As in Proposition~\ref{PropMainBranch}, we can rewrite this sum with the period $K_1 =  \frac{N_1}{n_0n_1}$ found in Lemma~\ref{LemmaVerSum} as \[\frac{\Lmb-1}{1-\Lmb^{-\nu_1}T^{N_1}}\sum_{r=0}^{\frac{N_1}{n_0n_1}-1}\sum_{m=(p-1)l_1}^{pl_1-1}\Lmb^{-(c_{1,p+r}((p+r)n_0n_1+m)+1)}T^{(p+r)n_0n_1 + m +1}.\] This is equal to the result in (1). The second part of the proposition follows from similar arguments. To prove (3), we can again use Lemma~\ref{LemmaVerSum} (note that $N_g = \frac{n_g\lbeta_g}{e_g} = n_g\lbeta_g$) for 
\begin{align*}
	&\sum_{\kappa \geq 0} \sum_{I_{g,(1+\kappa)n_2\cdots n_g}^{(1)}}(\Lmb-1)\Lmb^{-(c_{g,(1+\kappa)n_2\cdots n_g}(m)+1)}T^{m+1} \\
	& = \sum_{\kappa \geq 0} \sum_{m\geq 0} (\Lmb-1)\Lmb^{-(c_{g,(1 + \kappa)n_2\cdots n_g}((1 + \kappa)n_2\cdots n_g \frac{n_g\lbeta_g}{e_1} + m) + 1)}T^{(1 + \kappa)n_2\cdots n_g \frac{n_g\lbeta_g}{e_1} + m + 1}
\end{align*}
together with 
\begin{align*}
	\sum_{m\geq 0} \Lmb^{-(c_{g,n_2\cdots n_g}(n_2\ldots n_g \frac{n_g\lbeta_g}{e_1} + m) + 1)}T^{n_2\ldots n_g \frac{n_g\lbeta_g}{e_1} + m+1}  &= \sum_{m\geq 0} \Lmb^{-(\nu_g + g  + gm + 1)}T^{n_g\lbeta_g + m+1} \\
	&=  \frac{\Lmb^{-(\nu_g + g + 1)}T^{n_g\lbeta_g + 1}}{1-\Lmb^{-g}T}.
\end{align*}
\end{proof}

It remains to sum the first two expressions of Proposition~\ref{PropVerSum} horizontally over all $p \geq 1$. We again begin with a lemma that follows from simple computations. 
	
\begin{lem} \label{LemmaHorSum}
\begin{enumerate} [wide, labelindent=0pt]
	\item For all $k \geq 1, r\geq 0$ and $0 \leq m < (r+1)l_1 = (r+1)\big(\frac{n_2\lbeta_2}{e_1} - n_0n_1\big)$, we have 
	\begin{align*}
		\nu_2  =~& c_{1,k+ \frac{e_2N_2}{\lbeta_2} + r}\Big(\big(k + \frac{e_2N_2}{\lbeta_2} + r\big)n_0n_1 + \big(k + \frac{e_2N_2}{\lbeta_2} - 1\big)l_1 + m\Big) \\
		& - c_{1,k + r}((k+ r)n_0n_1 + (k-1)l_1 +m).
	\end{align*}
	\item  For $i = 2, \ldots, g-1$ and $k \geq 1$, we have for all $r \geq 0$ and $m \in \n$ in the interval $\big[0,(r+1)l_i[ = [0,(r+1)n_2\cdots n_i\big(\frac{n_{i+1}\lbeta_{i+1}}{e_1}-\frac{n_i\lbeta_i}{e_1}\big)\big[$ that 
	\begin{align*}
		\nu_{i+1} =~& c_{i, (k + \frac{e_{i+1}N_{i+1}}{\lbeta_{i+1}} + r)n_2\cdots n_i}\Big(\big(k + \frac{e_{i+1}N_{i+1}}{\lbeta_{i+1}}  + r\big)n_2\cdots n_i\frac{n_i\lbeta_i}{e_1} + \big(k + \frac{e_{i+1}N_{i+1}}{\lbeta_{i+1}} - 1\big)l_i + m\Big) \\ 
		& - c_{i, (k + r)n_2\cdots n_i}\Big((k + r)n_2\cdots n_i\frac{n_i\lbeta_i}{e_1} + (k - 1)l_i + m\Big).
	\end{align*}
\end{enumerate}
\end{lem}

Using similar arguments as in the proof of Proposition~\ref{PropVerSum} with this lemma, the horizontal summation leads to the next result. With the visualization of Figure~\ref{IntervalsPicture}, we can rephrase the first line as follows. In order to calculate the whole contribution of the first intervals $[kn_0n_1,\frac{kn_2\lbeta_2}{e_1}[\cap \n$ for $k \geq 1$, we only need to consider the block of intervals $I_{1,k}^{(p)}$ for $p = 1, \ldots, \frac{e_2N_2}{\lbeta_2}$ and the first $\frac{N_1}{n_0n_1}$ suitable $k$ (i.e., $k = p, \ldots, p+\frac{N_1}{n_0n_1}$). The second line can be interpreted in the same way, and the last line is the part of the infinite branches. 

\begin{prop} \label{PropHorSum}
The contribution to $J_Y(T)$ of the side branches for $k\geq 1$ is 
\begin{align*}
	&\frac{\Lmb-1}{(1-\Lmb^{-\nu_1}T^{N_1})(1-\Lmb^{-\nu_2}T^{N_2})}\sum_{r=0}^{\frac{N_1}{n_0n_1}-1}\sum_{r' = 1}^{\frac{e_2N_2}{\lbeta_2}} \sum_{I_{1,r'+r}^{(r')}}\Lmb^{-(c_{1,r'+r}(m)+1)}T^{m+1} \\
	& + \sum_{i=2}^{g-1}\frac{\Lmb - 1}{(1-\Lmb^{-\nu_i}T^{N_i})(1-\Lmb^{-\nu_{i+1}}T^{N_{i+1}})}\sum_{r=0}^{\frac{e_iN_i}{n_i\lbeta_i}-1} \sum_{r'=1}^{\frac{e_{i+1}N_{i+1}}{\lbeta_{i+1}}} \sum_{I_{i,(r'+r)n_2\cdots n_i}^{(r')}} \Lmb^{-(c_{i,(r'+r)n_2\cdots n_i}(m) + 1)}T^{m+1} \\
	& + \frac{(\Lmb-1)\Lmb^{-(\nu_g + g + 1)}T^{N_g + 1}}{(1-\Lmb^{-g}T)(1-\Lmb^{-\nu_g}T^{N_g})}.
\end{align*}
\end{prop}	
	
Combining all propositions of this section with the relation \[Z^{mot}_Y(T) = 1- \frac{1-T}{T}J_Y(T),\] we are now ready to give an explicit expression for the motivic zeta function of $Y \subset \co^{g+1}$. 

\begin{theo}\label{TheoremMotFunction}
Consider a space monomial curve $Y \subset \co^{g+1}$ defined by the equations~\eqref{EquationsY}. Let $N_i$ and $\nu_i$ for $i=1, \ldots, g$ be the positive integers defined as \[N_i := \lcm\Big(\frac{\lbeta_i}{e_i}, n_i,\ldots, n_g\Big),\] and \[\nu_i := N_i\left(\frac{1}{n_i\lbeta_i} \bigg(\sum_{l=0}^i\lbeta_l - \sum_{l=1}^{i-1}n_l\lbeta_l\bigg) + (i-1)+ \sum_{l=i+1}^g \frac{1}{n_l}\right).\] The motivic Igusa zeta function associated with $Y \subset \co^{g+1}$ is given by 
\begin{align*}
	Z^{mot}_Y(T) =~ & 1-(1-T)\Bigg(\frac{(\Lmb-1)\Lmb^{-(g+1)}}{1-\Lmb^{-g}T} + \frac{\Lmb^{-(g+1)}}{1- \Lmb^{-\nu_1}T^{N_1}}\sum_{r=0}^{N_1-1}\Lmb^{-\sum\limits_{i=0}^g [\frac{r}{n_i}]}T^r \\
	&+ \sum_{i=1}^{g-1}\frac{(\Lmb - 1)Z_i(T)}{(1-\Lmb^{-\nu_i}T^{N_i})(1-\Lmb^{-\nu_{i+1}}T^{N_{i+1}})} + \frac{(\Lmb-1)\Lmb^{-(\nu_g + g + 1)}T^{N_g}}{(1-\Lmb^{-g}T)(1-\Lmb^{-\nu_g}T^{N_g})}\Bigg).
\end{align*}
Here, $Z_i(T)$ for $i = 1, \ldots, g-1$ are polynomials with coefficients in $\z[\Lmb, \Lmb^{-1}]$. More precisely,
\begin{alignat*}{2}
	&Z_1(T) &&:= \sum_{r=0}^{\frac{N_1}{n_0n_1}-1}\sum_{r' = 1}^{\frac{e_2N_2}{\lbeta_2}} \sum_{I_{1,r'+r}^{(r')}}\Lmb^{-(c_{1,r'+r}(m)+1)}T^{m}, \\
	&Z_i(T) && := \sum_{r=0}^{\frac{e_iN_i}{n_i\lbeta_i}-1} \sum_{r'=1}^{\frac{e_{i+1}N_{i+1}}{\lbeta_{i+1}}} \sum_{I_{i,(r'+r)n_2\cdots n_i}^{(r')}} \Lmb^{-(c_{i,(r'+r)n_2\cdots n_i}(m) + 1)}T^{m}, \qquad i=2, \ldots, g-1,
\end{alignat*} 
where, for $i=1, \ldots, g-1$ and $k,m,p \in \n$, \[I_{i,k}^{(p)} := \bigg[\frac{kn_i\lbeta_i}{e_1} + (p-1)\Big(\frac{n_{i+1}\lbeta_{i+1}}{e_i} - \frac{n_i\lbeta_i}{e_i} \Big), \frac{kn_i\lbeta_i}{e_1} + p\Big(\frac{n_{i+1}\lbeta_{i+1}}{e_i} - \frac{n_i\lbeta_i}{e_i} \Big) \bigg[\cap \n,\] and \[c_{i,k}(m):= k(n_0+n_1)+\sum_{l=2}^i\frac{k\bar{\beta}_l}{e_1}+ \sum_{l=1}^i\Big(m-\frac{kn_l\bar{\beta}_l}{e_1}+1\Big)+\sum_{l=i+1}^g\Big(\Big[\frac{m}{n_l}\Big]+1\Big).\]
\end{theo}

We see that the motivic zeta function is indeed a rational function in $T$ and we can take a look at its poles. Since $\mathcal{M}_{\co}$ is not an integral domain, see for instance the appendix of~\cite{C}, one has to be careful with defining a pole. However, for example with the definition given in~\cite{RV}, one can see that a complete list of possible poles for $Z^{mot}_Y(T)$ is \[\Lmb^g, \qquad \Lmb^{\frac{\nu_i}{N_i}}, \quad i=1, \ldots, g,\] which could intuitively be expected from the above expression. In the next section, we will prove that all these candidates are actual poles, as we already notice in the following examples. \\

\begin{ex} \label{ExMotivic}
\begin{enumerate}[wide, labelindent=0pt]
	\item Consider the irreducible plane curve given by $(x_1^2-x_0^3)^2-x_0^5x_1 = 0$. Its semigroup has $(4,6,13)$ as unique minimal set of generators, and the corresponding space curve $Y_1 \subseteq \co ^3$ in three variables $(g=2$) is defined by 
	\[\left\{ 
	\begin{array}{r c l l}
		x_1^2 & - & x_0^3  &  = 0 \\
		x_2^2  &- & x_0^5x_1 &= 0.  \\
 	\end{array}
 	\right.\]
	Using Theorem~\ref{TheoremMotFunction}, one can compute that \[Z^{mot}_{Y_1}(T) = \frac{(\Lmb-1)P_1(T)}{\Lmb^{47}(1-\Lmb^{-2}T)(1-\Lmb^{-8}T^6)(1-\Lmb^{-37}T^{26})}\] where $P_1(T)$ is the polynomial
 	\begin{align*}
		& (\Lmb+1 ) {T}^{31}-{\Lmb}^{3}{T}^{30}+{\Lmb}^{3}{T}^{29}- ({\Lmb}^{6}+{\Lmb}^{5} ) {T}^{28}+ ( {\Lmb}^{6}+{\Lmb}^{5} ) {T}^{27}-2\,{\Lmb}^{8}{T}^{26} + ( -{\Lmb}^{9}+{\Lmb}^{8} ) {T}^{25} \\
		& +( {\Lmb}^{12}-{\Lmb}^{11} ) {T}^{24}+ ( -{\Lmb}^{12}+{\Lmb}^{11}) {T}^{23}+ ( {\Lmb}^{15}-{\Lmb}^{14} ) {T}^{22}+ ( - {\Lmb}^{15}+{\Lmb}^{14} ) {T}^{21}+{\Lmb}^{18}{T}^{20}\\
		& -{\Lmb}^{18}{T}^{19} -{\Lmb}^{25}{T}^{14}+{\Lmb}^{25}{T}^{13}+ ( {\Lmb}^{29}-{\Lmb}^{28} ) {T}^{12} + ( -{\Lmb}^{29}+{\Lmb}^{28} ) {T}^{11}+ ( {\Lmb}^{32}-{\Lmb}^{31} ) {T}^{10}\\
		& + ( -{\Lmb}^{32}+{\Lmb}^{31} ) {T}^{9}+( {\Lmb}^{35}-{\Lmb}^{34} ) {T}^{8} + ( -{\Lmb}^{35}+{\Lmb}^{34}) {T}^{7}- ( {\Lmb}^{38}+{\Lmb}^{37} ) {T}^{5}+{\Lmb}^{40}{T}^{4}\\
		& -{\Lmb}^{40}{T}^{3}+ ( {\Lmb}^{43}+{\Lmb}^{42} ) {T}^{2} -( {\Lmb}^{43}+{\Lmb}^{42} ) T+{\Lmb}^{46}+{\Lmb}^{45}.
	\end{align*}
	This expression has three poles, $\Lmb^2,\Lmb^{\frac{8}{6}}$ and $\Lmb^{\frac{37}{26}}$, corresponding to $g = 2, \frac{\nu_1}{N_1} = \frac{8}{6}$ and $\frac{\nu_2}{N_2} = \frac{37}{26}$, respectively. These are precisely the set of candidate poles from Theorem~\ref{TheoremMotFunction}.
	\item The polynomial $((x_1^2-x_0^3)^2 - x_0^5x_1)^2 - x_0^{10}(x_1^2-x_0^3)$ defines an irreducible plane curve whose semigroup is minimally generated by $(8,12,26,53)$, and it induces the curve $Y_2 \subseteq \co^4$ given by
	\[\left\{ 
	\begin{array}{r c l l}
		x_1^2 & - & x_0^3  &  = 0 \\
		x_2^2  &- & x_0^5x_1 &= 0 \\
		x_3^2 & - &x_0^{10}x_2 &=0.
 	\end{array}\right.\]  
	For this curve, Theorem~\ref{TheoremMotFunction} gives \[Z^{mot}_{Y_2}(T) = \frac{(\Lmb-1)P_2(T)}{\Lmb^{299}(1-\Lmb^{-3}T)(1-\Lmb^{-11}T^6)(1-\Lmb^{-50}T^{26})(1-\Lmb^{-235}T^{106})},\] for a concrete polynomial $P_2(T)$ of degree $137$ with coefficients in $\z[\Lmb]$, which occupies more than half a page. Again, all candidate poles turn out to be actual poles.  
\end{enumerate}
\end{ex}

Both results in Example~\ref{ExMotivic} were also obtained using other methods in~\cite[Chapter 7]{P}; there, the local $p$-adic zeta function of $Y_i$ was calculated in terms of a principalization of its defining ideal. From the data of the same principalization, one can deduce an expression for the global $p$-adic zeta function of $Y_i$, to which the above expression for the global motivic zeta function specializes.

\begin{rem}\label{RemPrinc}
For a general monomial curve $Y \subset \co^{g+1}$, we do not see how to construct an explicit principalization of its defining ideal in order to deduce the motivic Igusa zeta function from it. In addition, we would like to point out that the resolution constructed in~\cite[Section 5]{MVV} is also not sufficient to compute the motivic zeta function of $Y$. More precisely, in~\cite{MVV}, the problem of studying the monodromy eigenvalues of $Y$ is handled by considering $Y$ as a Cartier divisor on a generic embedding surface $S \subset \co^{g+1}$ and constructing an embedded $\mathbb Q$-resolution of $Y \subset S$. This $\mathbb Q$-resolution could be used to compute a part of the motivic zeta function of $Y$, but not its whole zeta function.
\end{rem}

The approach in this section also provides a way to compute the local motivic zeta function associated with a space monomial curve $Y \subset \co^{g+1}$. As in Section~\ref{JetMotivic}, one can check that the relations $[\mathcal{X}_{0,0}] = \emptyset$ and $[\mathcal{X}_{m,0}] = \Lmb^{g+1}[\pi^{-1}_{m-1}(0)_{red}] - [\pi^{-1}_m(0)_{red}]$ for $m \geq 1$ imply that \[\Lmb^{-(g+1)}\sum_{m\geq 0}[\mathcal{X}_{m,0}](\Lmb^{-(g+1)}T)^m= \Lmb^{-(g+1)}-\frac{1-T}{T}\sum_{m\geq0}[\pi^{-1}_m(0)_{red}](\Lmb^{-(g+1)}T)^{m+1}.\] Therefore, the local version is equal to the above expression with the first $1$ replaced by $\Lmb^{-(g+1)}$ and without the term \[-\frac{(1-T)(\Lmb-1)\Lmb^{-(g+1)}}{1-\Lmb^{-g}T},\] which comes from the side branch of $0$ consisting of $[Y_m \setminus \pi^{-1}_m(0)_{red}]$. \\

For $g=1$, we can repeat most steps of these computations: we can change the stratification~\eqref{StratificationForg=1} of $\pi^{-1}_m(0)_{red}$ in exactly the same way such that $\mathcal{B}_m \simeq \co^{(g+1)(m+1) - c(m)}$ for every $m \in \n$; we can split the calculations in a side branch of $0$ with $\pi^{-1}_m(Y\setminus \{0\})$ for all $m \geq 0$, a main branch containing $\mathcal{B}_m$ for all $m \geq 0$, and side branches at $kn_0n_1$ for all $k \geq 1$; and we get the same results for the side branch of $0$ and the main branch as in Proposition~\ref{PropAboveReg} and Proposition~\ref{PropMainBranch}, respectively. The only difference is that each $k \geq 1$ has an infinite branch consisting of $D_{m,k}$ with codimension $c_{1,k}(m)$ for all $m \geq kn_0n_1$. However, this can be treated similarly as the infinite branches for $g \geq 2$, and we obtain the same expression as in Proposition~\ref{PropVerSum}, part (3). In other words, the motivic zeta function of the plane curve $Y = V(x_1^{n_1}-x_0^{n_0}) \subset \co^2$ is given by the expression in Theorem~\ref{TheoremMotFunction} with $g=1$. 

\section{Poles of the motivic zeta function of a space monomial curve}\label{PolesMotivic}

The explicit expression for the motivic Igusa zeta function associated with a space monomial curve $Y \subset \co^{g+1}$ in Theorem~\ref{TheoremMotFunction} provides the following $g+1$ candidate poles for all $g \geq 1$: \[\Lmb^g, \qquad \Lmb^{\frac{\nu_i}{N_i}}, \quad i=1, \ldots, g.\] We will now show that all these possible poles are actual poles. \\

Instead of proving this for the motivic zeta function directly, we will work with the \emph{topological Igusa zeta function} associated with $Y$. This zeta function was first introduced by Denef and Loeser~\cite{DL1} for one polynomial $f$ in terms of an embedded resolution of $f$. Such  a resolution can also be used to express the motivic zeta function of $f$ and to show that this function specializes to the topological one, see for example~\cite{DL2}. In particular, a pole of the topological zeta function induces a pole of the motivic zeta function. For an ideal, one can obtain a similar formula in terms of a principalization of the ideal, where the topological version is again a specialization of the motivic one. The generalization to ideals by using a principalization is mentioned in~\cite{VZ}. \\

Roughly speaking, we obtain the topological zeta function $Z^{top}_Y(s)$ for $s \in \co$ by substituting $T = \Lmb^{-s}$ in $Z^{mot}_Y(T)$ and taking the limit $\Lmb \rightarrow 1$. Formally, one should first specialize to the \emph{Hodge zeta function}, using Hodge polynomials, and then to the topological zeta function. In this way, we get the following expression for the topological zeta function associated with a space monomial curve $Y \subset \co^{g+1}$: 
\begin{align*}
	Z^{top}_Y(s) = ~& \frac{\nu_1}{\nu_1 + sN_1} - \sum_{i=1}^{g-1}\frac{e_{i+1}N_iN_{i+1}}{n_i\lbeta_i\lbeta_{i+1}}(n_{i+1}\lbeta_{i+1} - n_i\lbeta_i)\frac{s}{(\nu_i + sN_i)(\nu_{i+1} +s N_{i+1})} \\
	&- \frac{s}{(g + s)(\nu_g + sN_g)}.
\end{align*}
The candidate poles of this rational function are clearly $-g$ and $-\frac{\nu_i}{N_i}$ for $i=1, \ldots, g,$ which correspond to the possible poles for the motivic zeta function. Therefore, we will prove the stronger result that each of these candidates is an actual pole of $Z^{top}_Y(s)$. 

\begin{ex} 
We consider the two curves $Y_1$ and $Y_2$ from Example~\ref{ExMotivic}. Applying the above specialization yields \[Z^{top}_{Y_1}(s) = \frac{4(47s^2 + 169s + 148)}{(2 + s)(8 + 6s)(37 + 26s)},\] and \[Z^{top}_{Y_2}(s) = \frac{2(14176s^3+103282s^2+246789s+193875)}{(3 + s)(11 + 6s)(50 + 26s)(235 + 106s)},\] from which we clearly see that all candidate poles are actual poles. Again, both these topological zeta functions were also found using a principalization in~\cite[Chapter 7]{P}.
\end{ex}

We start with remarking the following inequalities between the candidate poles.

\begin{lem} \label{LemmaRelPoles}
The candidate poles can be ordered as \[-g < -\frac{\nu_g}{N_g} < -\frac{\nu_{g-1}}{N_{g-1}} < \cdots < -\frac{\nu_1}{N_1}.\] In particular, this implies that every candidate pole is possibly an actual pole of order $1$.
\end{lem} 

\begin{proof} 
Let $g \geq 2$ and take $i \in \{1, \ldots, g-1\}$ fixed. The difference between $\frac{\nu_{i+1}}{N_{i+1}}$ and $\frac{\nu_i}{N_i}$ can be rewritten as 
\begin{equation} \label{DifferencePoles}
	\frac{\nu_{i+1}}{N_{i+1}} - \frac{\nu_i}{N_i}  = \bigg(\frac{1}{n_i\lbeta_i} - \frac{1}{n_{i+1}\lbeta_{i+1}}\bigg)\bigg(- \lbeta_0 + \sum_{l=1}^i(n_l-1)\lbeta_l \bigg).
\end{equation}
From the proof of Corollary~\ref{CorIrrComWhole}, we know that $-\lbeta_0 + (n_1-1)\lbeta_1 > 0$. Together with $n_i\lbeta_i < \lbeta_{i+1}$ and $n_l>1$ for all $l=1, \ldots, i$, we indeed see that $\frac{\nu_{i+1}}{N_{i+1}} - \frac{\nu_i}{N_i} > 0.$ Similarly, both the inequality $-g < -\frac{\nu_g}{N_g}$ and the case for $g = 1$ follow from the positive difference \[g - \frac{\nu_g}{N_g} = \frac{1}{n_g\lbeta_g}\bigg(- \lbeta_0 + \sum_{l=1}^g(n_l-1)\lbeta_l \bigg). \qedhere\]
\end{proof}

This relation between the candidate poles immediately yields the \emph{log canonical threshold} of the pair $(\co^{g+1},Y)$. The log canonical threshold is an important invariant in birational geometry, and we refer to~\cite{K} or~\cite{Mu2} for more about it. See also Remark~\ref{RemDimLct}. 

\begin{cor}\label{lct}
The log canonical threshold of the pair $(\co^{g+1},Y)$ is $\frac{\nu_1}{N_1} = \sum\limits_{l=0}^g \frac{1}{n_l}$.
\end{cor}

We can now state and prove the main result of this section.

\begin{theo}\label{ThmPoles}
A complete list of the poles of the topological Igusa zeta function associated with a space monomial curve $Y \subset \co^{g+1}$ is given by \[-g, \qquad -\frac{\nu_i}{N_i} = -\left(\frac{1}{n_i\lbeta_i} \bigg(\sum_{l=0}^i\lbeta_l - \sum_{l=1}^{i-1}n_l\lbeta_l\bigg) + (i-1)+ \sum_{l=i+1}^g \frac{1}{n_l}\right), \quad i=1, \ldots, g,\] and all these poles have order $1$. Consequently, the motivic Igusa zeta function associated with $Y$ has poles \[\Lmb^g, \qquad \Lmb^{\frac{\nu_i}{N_i}}, \quad i=1, \ldots, g,\] which are all poles of order $1$.
\end{theo}

\begin{proof}
We will show that the residue of each candidate pole is non-zero. For every $g \geq 1$, this is trivial for the residue of the smallest pole, $-g$, given by \[\frac{g}{\nu_g - gN_g}.\] To investigate the remaining residues, we denote the residue corresponding to $-\frac{\nu_i}{N_i}$ by $R_i$ for $i = 1, \ldots, g$, and we list them, up to a factor of $\frac{\nu_i}{N_i^2}$, in the next table.
{\tabulinesep=1.3mm
\begin{longtabu}{| >{\centering}p{.1\textwidth} | p{.75\textwidth} |}  
	\hline
	$g$ & residues \\
	\hline \hline
	$g = 1$ & $R_1= N_1 + \frac{1}{g-\frac{\nu_1}{N_1}}$ \\
	\hline
	\multirow{3}{*}{$g \geq 2$} &$R_1 = N_1 + \frac{e_2N_1}{n_1\lbeta_1\lbeta_2}(n_2\lbeta_2- n_1\lbeta_1)\frac{1}{\frac{\nu_2}{N_2}-\frac{\nu_1}{N_1}}$  \\
	& $R_{i, i = 2, \ldots, g-1} = \frac{e_iN_i}{n_{i-1}\lbeta_{i-1}\lbeta_i}(n_i\lbeta_i - n_{i-1}\lbeta_{i-1})\frac{1}{\frac{\nu_{i-1}}{N_{i-1}}-\frac{\nu_i}{N_i}}$ \\ 
	& \hspace{68pt} + $\frac{e_{i+1}N_i}{n_i\lbeta_i\lbeta_{i+1}}(n_{i+1}\lbeta_{i+1} - n_i\lbeta_i)\frac{1}{\frac{\nu_{i+1}}{N_{i+1}}-\frac{\nu_i}{N_i}}$,  \\ 
	& $R_g = \frac{N_g}{n_{g-1}\lbeta_{g-1}\lbeta_g}(n_g\lbeta_g - n_{g-1}\lbeta_{g-1})\frac{1}{\frac{\nu_{g-1}}{N_{g-1}}-\frac{\nu_g}{N_g}} + \frac{1}{g-\frac{\nu_g}{N_g}}$ \\
	\hline 
\end{longtabu}}
From the relation between the candidate poles in Lemma~\ref{LemmaRelPoles}, it immediately follows for all $g$ that $R_1 > 0$, being the sum of two positive numbers. We claim that all the other residues for $g\geq 2$ are strictly negative, which is less trivial as they consist of a negative and a positive part. We take a look at $R_i$ for $i = 2, \ldots, g-1$ and $g \geq 3$; the last residue $R_g$ for $g \geq 2$ can be treated in a similar way. The inequality $R_i < 0$ is equivalent to \[\frac{e_i}{n_{i-1}\lbeta_{i-1}}(n_i\lbeta_i - n_{i-1}\lbeta_{i-1})\Big(\frac{\nu_{i+1}}{N_{i+1}}-\frac{\nu_i}{N_i}\Big) > \frac{e_{i+1}}{n_i\lbeta_{i+1}}(n_{i+1}\lbeta_{i+1} - n_i\lbeta_i)\Big(\frac{\nu_i}{N_i} - \frac{\nu_{i-1}}{N_{i-1}}\Big).\] Using formula~\eqref{DifferencePoles} from the proof of Lemma~\ref{LemmaRelPoles}, this can be rewritten as
\begin{align*} 
	&\frac{e_i}{n_{i-1}\lbeta_{i-1}}(n_i\lbeta_i - n_{i-1}\lbeta_{i-1})\Big(\frac{1}{n_i\lbeta_i} - \frac{1}{n_{i+1}\lbeta_{i+1}}\Big)\Big(-\lbeta_0 + \sum_{l=1}^i(n_l - 1)\lbeta_l\Big) \\
	& > \frac{e_{i+1}}{n_i\lbeta_{i+1}}(n_{i+1}\lbeta_{i+1} - n_i\lbeta_i)\Big(\frac{1}{n_{i-1}\lbeta_{i-1}} - \frac{1}{n_i\lbeta_i}\Big)\Big(-\lbeta_0 + \sum_{l=1}^{i-1}(n_l - 1)\lbeta_l\Big).
\end{align*} 
Finally, multiplying both sides by $\frac{n_{i-1}n_in_{i+1}\lbeta_{i-1}\lbeta_i\lbeta_{i+1}}{e_i} = \frac{n_{i-1}n_i\lbeta_{i-1}\lbeta_i\lbeta_{i+1}}{e_{i+1}}  $ gives the condition
\begin{align*}
	&(n_i\lbeta_i - n_{i-1}\lbeta_{i-1})(n_{i+1}\lbeta_{i+1} - n_i\lbeta_i )\Big(-\lbeta_0 + \sum_{l=1}^i(n_l - 1)\lbeta_l\Big) \\ 
	& >  \frac{1}{n_i}(n_{i+1}\lbeta_{i+1} - n_i\lbeta_i)(n_i\lbeta_i - n_{i-1}\lbeta_{i-1} )\Big(-\lbeta_0 + \sum_{l=1}^{i-1}(n_l - 1)\lbeta_l\Big),
\end{align*}
	which is easily seen to hold. 
\end{proof}

\begin{rem}
Applying the same specialization to the local version of the motivic zeta function, one obtains a local version for the topological zeta function. Because the limit for $\Lmb \rightarrow 1$ of \[\frac{(1-T)(\Lmb-1)\Lmb^{-(g+1)}}{1-\Lmb^{-g}T}\vert_{T = \Lmb^{-s}}\] is equal to $0$, the global and local topological zeta function of $Y$ are identical. Hence, the results in this section are true for both the global and the local motivic zeta function.
\end{rem}

Theorem~\ref{ThmPoles} implies in particular that the motivic (resp. topological) zeta function of the special fiber $Y$ of the family $\eta:(\chi,0) \rightarrow (\co,0)$ has the same number of poles as the motivic (resp. topological) zeta function of a generic fiber in the family, whose poles are equal to the poles associated with the plane branch times a factor $\Lmb^{g-1}$ (resp. with an integer shift of $-(g-1)$). This is intriguing as the induced family on the jet schemes is in most cases not flat by Theorem~\ref{TheoremNotFlat}, and the motivic zeta function is calculated in terms of the codimensions of the (irreducible components of) the jet schemes. The poles, however, are in general not equal; one can check, using for instance the expressions in~\cite[Section 2]{NV}, that only the poles $\Lmb^g$ and $\Lmb^{\frac{\nu_g}{N_g}}$ of the motivic (resp. $-g$ and $-\frac{\nu_g}{N_g}$ of the topological) zeta function associated with $Y$ are also poles associated with a generic fiber.

\end{document}